\documentclass[11pt,reqno]{amsart}

\usepackage{amsmath,amsthm,amssymb}
\usepackage{tikz}
\usepackage{ytableau}
\ytableausetup{centertableaux,boxsize=1.2em}
\usepackage{capt-of}
\usepackage[hidelinks]{hyperref}
\hypersetup{pdftitle={Rectangular Symmetry for Catalan Functions},pdfauthor={Hyeonjae Choi}}
\usepackage[top=1in,bottom=1in,left=1in,right=1in]{geometry}
\usepackage{cleveref}
\usepackage[backend=biber,bibstyle=trad-plain,citestyle=numeric,sorting=nyt,giveninits=false,maxbibnames=99,doi=false,url=false,isbn=false,eprint=false]{biblatex}
\DefineBibliographyExtras{english}{}
\AtEveryBibitem{\ifentrytype{book}{\clearfield{pages}\clearfield{pagetotal}}{}}
\numberwithin{equation}{section}

\newtheorem{thm}{Theorem}[section]
\newtheorem{lem}[thm]{Lemma}
\newtheorem{prop}[thm]{Proposition}

\theoremstyle{definition}
\newtheorem{defn}[thm]{Definition}
\newtheorem{exmp}[thm]{Example}
\theoremstyle{remark}
\newtheorem{rmk}[thm]{Remark}

\crefname{thm}{Theorem}{Theorems}
\crefname{lem}{Lemma}{Lemmas}
\crefname{prop}{Proposition}{Propositions}
\crefname{cor}{Corollary}{Corollaries}
\crefname{defn}{Definition}{Definitions}
\crefname{exmp}{Example}{Examples}
\crefname{rmk}{Remark}{Remarks}
\crefname{section}{Section}{Sections}
\crefname{subsection}{Section}{Sections}
\crefname{equation}{Equation}{Equations}
\crefname{figure}{Figure}{Figures}
\Crefname{thm}{Theorem}{Theorems}
\Crefname{lem}{Lemma}{Lemmas}
\Crefname{prop}{Proposition}{Propositions}
\Crefname{cor}{Corollary}{Corollaries}
\Crefname{defn}{Definition}{Definitions}
\Crefname{exmp}{Example}{Examples}
\Crefname{rmk}{Remark}{Remarks}
\Crefname{section}{Section}{Sections}
\Crefname{subsection}{Section}{Sections}
\Crefname{equation}{Equation}{Equations}
\Crefname{figure}{Figure}{Figures}
\crefformat{equation}{(#2#1#3)}
\Crefformat{equation}{Equation #2(#1)#3}
\AddToHook{env/lem/begin}{\crefalias{thm}{lem}}
\AddToHook{env/prop/begin}{\crefalias{thm}{prop}}
\AddToHook{env/cor/begin}{\crefalias{thm}{cor}}
\AddToHook{env/exmp/begin}{\crefalias{thm}{exmp}}
\AddToHook{env/defn/begin}{\crefalias{thm}{defn}}
\AddToHook{env/rmk/begin}{\crefalias{thm}{rmk}}

\newcommand{\SSYT}{\operatorname{SSYT}}
\newcommand{\charge}{\operatorname{charge}}
\newcommand{\cocharge}{\operatorname{cocharge}}
\newcommand{\rw}{\operatorname{rw}}
\newcommand{\width}{\operatorname{width}}
\newcommand{\kat}{\operatorname{kat}}
\newcommand{\cat}{\operatorname{cat}}
\newcommand{\rot}{\operatorname{rot}}
\newcommand{\sgn}{\operatorname{sgn}}
\newcommand{\inv}{\operatorname{inv}}
\newcommand{\dec}{\operatorname{dec}}
\newcommand{\wt}{\operatorname{wt}}
\newcommand{\sh}{\operatorname{sh}}
\newcommand{\evec}{\mathbf{e}}
\newcommand{\dvec}{\mathbf{d}}
\newcommand{\tset}{\mathcal{T}}
\newcommand{\F}{\mathcal{F}}
\newcommand{\id}{\mathrm{id}}

\title[Rectangular symmetry for Catalan functions]{Rectangular Symmetry for Catalan Functions}

\author{Hyeonjae Choi}
\address{Department of Mathematics and Institute of Basic Science \\ Sungkyunkwan University \\Suwon \\ Korea}
\email{hyeonjaechoi@skku.edu}

\subjclass[2020]{Primary 05E05, 05E10; Secondary 05A30, 17B37}
\keywords{Catalan functions, root ideals, katabolizable tableaux, rectangular symmetry, column complementation, combinatorial $R$-matrix}

\date{}

\begin{document}

\begin{abstract}
Katabolizable tableaux provide a combinatorial formula for the Schur coefficients of Catalan functions.
We prove that a tableau involution preserves katabolizability under antidiagonal reflection of any upper root ideal.
This involution complements shape and content in a fixed rectangle and preserves charge.
We obtain a bijective proof of the rectangular symmetry of these coefficients.
\end{abstract}

\maketitle

\section{Introduction}\label{sec:intro}

Catalan functions were studied by Blasiak, Morse, Pun, and Summers in their work on $k$-Schur functions~\cite{BlaMorPunSum2019}.
They generalize modified Hall-Littlewood functions and have a geometric interpretation in terms of $\mathrm{GL}_n$-equivariant Euler characteristics on the flag variety~\cite[Theorem~3.1]{BlaMorPunSum2019}.

Let an upper root ideal $\Psi \subseteq \Delta_n^+ = \{(i, j) : 1 \leq i < j \leq n\}$ and a partition $\mu = (\mu_1, \ldots, \mu_n)$ be given.
We define the Catalan function by
\begin{equation}\label{eq:catalan-definition}
	H(\Psi; \mu) = \prod_{(i, j) \in \Psi}(1 - qR_{i, j})^{-1}s_\mu,
\end{equation}
where $s_\mu$ is the Schur function.
For $\alpha \in \mathbb Z^n$, $s_\alpha$ is defined by the Jacobi-Trudi formula.
The raising operator $R_{i, j}$ acts on the subscript by $\alpha \mapsto \alpha + \evec_i - \evec_j$, where $\evec_i$ is the $i$th coordinate vector.
All index shifts are applied before evaluating the Jacobi-Trudi determinant.

We define $K^\Psi_{\lambda, \mu}(q)$ as the coefficient of $s_\lambda$ in the Catalan function $H(\Psi; \mu)$, i.e., $H(\Psi; \mu) = \sum_\lambda K^\Psi_{\lambda, \mu}(q)s_\lambda$.
Blasiak, Morse, and Pun gave the combinatorial formula~\cite[Theorem~2.18]{BlaMorPun2024}
\[
	K^\Psi_{\lambda, \mu}(q) = \sum_{T \in \tset_\Psi(\lambda, \mu)}q^{\charge(T)},
\]
where $\tset_\Psi(\lambda, \mu)$ consists of semistandard tableaux of shape $\lambda$ and content $\mu$ satisfying the \emph{$\Psi$-katabolizability} condition, which we recall in \cref{subsec:catalan}.
Here, $\charge$ is the Lascoux-Sch\"utzenberger charge statistic~\cite{LasSch1978}.

We consider partitions $\lambda, \mu \subseteq (k^n)$ with $|\lambda| = |\mu|$ and append zeros so that both have $n$ parts.
The rectangular complements $\lambda^\perp$ and $\mu^\perp$ are defined by
\[
	\lambda^\perp = (k - \lambda_n, \ldots, k - \lambda_1),
	\qquad
	\mu^\perp = (k - \mu_n, \ldots, k - \mu_1).
\]
We view $\Delta_n^+$ as the boxes above the diagonal of an $n \times n$ matrix and write
\[
	\Psi^\dagger = \{(n + 1 - j, n + 1 - i) : (i, j) \in \Psi\}
\]
for the reflection of $\Psi$ in the antidiagonal of the matrix.

The coefficients $K^\Psi_{\lambda, \mu}(q)$ satisfy the following rectangular symmetry:
\begin{equation}\label{eq:intro-catalan-symmetry}
	K^\Psi_{\lambda, \mu}(q)
		= K^{\Psi^\dagger}_{\lambda^\perp, \mu^\perp}(q).
\end{equation}
Shimozono and Weyman proved this identity for parabolic root ideals, whose boxes are exactly those above a block diagonal~\cite[Section~2.8, (2.16)]{ShiWey2000}.
Their algebraic argument also applies to arbitrary upper root ideals.

For the full root ideal, $K^{\Delta_n^+}_{\lambda, \mu}(q) = K_{\lambda, \mu}(q)$, so this identity reduces to the rectangular symmetry of Kostka-Foulkes polynomials.
Briand, Orellana, and Rosas gave a proof using symmetric functions and studied related symmetries for other coefficients~\cite{BriOreRos2015}.

\emph{Column complementation} in the rectangle $(k^n)$ replaces each of the $k$ columns by its missing letters from $[n] = \{1, \ldots, n\}$, listed in increasing order, and reverses the column order~\cite[Section~3.3]{LasSch1981}.
The Lascoux-Sch\"utzenberger (LS) action of the longest permutation $w_0(i) = n + 1 - i$ reverses the content vector and preserves shape.
We define $\Phi_{n, k}$ by column complementation followed by the LS action of $w_0$.
This involution complements shape and content in $(k^n)$ and preserves charge~\cite[Property~5.3(iii)]{LasSch1981}.

When $\Psi$ is a parabolic root ideal and $\mu$ is constant on each block, Kirillov and Shimozono proved that column complementation preserves column catabolizability~\cite[Section~8, Proposition~10]{KirShi2002}.
For an arbitrary upper root ideal $\Psi$, we prove that $\Phi_{n, k}$ restricts to a bijection between $\tset_\Psi(\lambda, \mu)$ and $\tset_{\Psi^\dagger}(\lambda^\perp, \mu^\perp)$.

\begin{thm}\label{thm:main}
Let $\Psi \subseteq \Delta_n^+$ be an upper root ideal and $T$ a semistandard tableau with entries in $[n]$, partition content $\mu$, and width at most $k$.
Then $T$ is $\Psi$-katabolizable if and only if $\Phi_{n, k}T$ is $\Psi^\dagger$-katabolizable.
\end{thm}

As a corollary, we obtain a combinatorial proof of \eqref{eq:intro-catalan-symmetry} using charge preservation.

\subsection{Organization}

This paper is organized as follows.
In \cref{sec:algebraic-symmetry}, we give an algebraic proof of \eqref{eq:intro-catalan-symmetry}.
In \cref{sec:tableaux}, we recall the column complementation map and give an alternative proof of charge preservation.
In \cref{sec:katabolizability}, we recall katabolizable tableaux and their correspondence with classical highest-weight elements of DARK crystals.
In \cref{sec:dark-proof}, we prove \cref{thm:main} first for strictly decreasing content and then extend it to arbitrary partition content.

\section{Algebraic proof of rectangular symmetry}\label{sec:algebraic-symmetry}

In this paper, we assume $n \geq 1$ and write $[m] = \{1, \ldots, m\}$ for $m \geq 0$.
We identify each positive root $\evec_i - \evec_j$ of type $A_{n - 1}$ with the pair $(i, j)$, where $\evec_i$ is the $i$th coordinate vector.
We write $\Delta_n^+ = \{(i, j) : 1 \leq i < j \leq n\}$ for the set of positive roots.
An \emph{upper root ideal} is a subset $\Psi \subseteq \Delta_n^+$ such that $(i, j) \in \Psi$ implies $(r, s) \in \Psi$ whenever $1 \leq r \leq i < j \leq s \leq n$.
All root ideals in this paper are upper root ideals.

The antidiagonal reflection of $\Psi$ is $\Psi^\dagger = \{(n + 1 - j, n + 1 - i) : (i, j) \in \Psi\}$.
It is again an upper root ideal, and $(\Psi^\dagger)^\dagger = \Psi$.
For $k \geq 0$ and a partition $\lambda \subseteq (k^n)$, we append zeros to give $n$ parts.
We define its rectangular complement by $\lambda^\perp = (k - \lambda_n, \ldots, k - \lambda_1)$.
The longest element $w_0$ of $\mathfrak S_n$ acts on $\mathbb Z^n$ by reversing coordinates.
We write $1^n = (1, \ldots, 1)$.

For $\alpha \in \mathbb Z^n$, we define $s_{\alpha} = \det(h_{\alpha_i - i + j})_{1 \leq i, j \leq n}$, where $h_r$ is the complete homogeneous symmetric function of degree $r$, with $h_0 = 1$ and $h_r = 0$ for $r < 0$.
In \eqref{eq:catalan-definition}, we apply all index shifts before evaluating this determinant.
The Jacobi-Trudi formula gives the straightening rule $s_{\alpha} = \sgn(w)s_\lambda$ whenever $\alpha + \rho = w(\lambda + \rho)$ for a partition $\lambda$ and $w \in \mathfrak S_n$, where $\rho = (n - 1, \ldots, 0)$.
If no such partition exists, then $s_{\alpha} = 0$.

\begin{prop}\label{prop:algebraic-duality}
For an upper root ideal $\Psi \subseteq \Delta_n^+$ and partitions $\lambda, \mu \subseteq (k^n)$ with $|\lambda| = |\mu|$,
we have $K^\Psi_{\lambda, \mu}(q) = K^{\Psi^\dagger}_{\lambda^\perp, \mu^\perp}(q)$.
\end{prop}
\begin{proof}
For $\beta \in \mathbb Z^n$, we write $x^{\beta} = x_1^{\beta_1}\cdots x_n^{\beta_n}$ and define $\mathcal P_\Psi(\beta; q)$ by
\begin{equation*}
	\prod_{(i, j) \in \Psi}(1 - qx_i/x_j)^{-1}
	= \sum_{\beta \in \mathbb Z^n}\mathcal P_\Psi(\beta; q)x^{\beta}.
\end{equation*}
Since $H(\Psi; \mu) = \sum_{\beta \in \mathbb Z^n}\mathcal P_\Psi(\beta; q)s_{\mu + \beta}$, the straightening rule gives
\begin{equation}\label{eq:alternating-formula}
	K^\Psi_{\lambda, \mu}(q)
	= \sum_{w \in \mathfrak S_n}\sgn(w)\,
		\mathcal P_\Psi\bigl(w(\lambda + \rho) - (\mu + \rho); q\bigr).
\end{equation}
The map $-w_0$ sends each root $\evec_i - \evec_j$ of $\Psi$ to the root $\evec_{n + 1 - j} - \evec_{n + 1 - i}$ of $\Psi^\dagger$.
This map induces a bijection between root decompositions and preserves the number of roots, so
\begin{equation}\label{eq:partition-reflection}
	\mathcal P_{\Psi^\dagger}(-w_0\beta; q) = \mathcal P_\Psi(\beta; q).
\end{equation}
For a partition $\lambda \subseteq (k^n)$, we have
\[
	\lambda^\perp + \rho = (k + n - 1)1^n - w_0(\lambda + \rho).
\]
Thus, for $w \in \mathfrak S_n$,
\[
	(w_0ww_0)(\lambda^\perp + \rho) - (\mu^\perp + \rho)
	= -w_0\bigl(w(\lambda + \rho) - (\mu + \rho)\bigr).
\]
Since $w \mapsto w_0ww_0$ is a sign-preserving bijection of $\mathfrak S_n$, equations
\eqref{eq:alternating-formula} and \eqref{eq:partition-reflection}
imply $K^\Psi_{\lambda, \mu}(q) = K^{\Psi^\dagger}_{\lambda^\perp, \mu^\perp}(q)$.
\end{proof}
\begin{rmk}
\Cref{prop:algebraic-duality} holds for any subset $\Psi \subseteq \Delta_n^+$, with $H(\Psi; \mu)$ defined by \eqref{eq:catalan-definition}.
\end{rmk}

\section{The involution \texorpdfstring{$\Phi_{n, k}$}{Phi(n,k)} and charge}\label{sec:tableaux}

Lascoux and Sch\"utzenberger~\cite[Property~5.3(iii)]{LasSch1981} proved charge preservation under column complementation using the plactic monoid.
We give an alternative proof for tableaux in an arbitrary rectangle using the properties of cocharge.

\subsection{Tableaux and insertion}\label{subsec:insertion}

We use English notation for Young diagrams.
A \emph{semistandard Young tableau} (SSYT) is a filling of a Young diagram with weakly increasing rows and strictly increasing columns.
Throughout this paper, tableaux are semistandard.
In this section, their entries lie in $[n]$.
For a partition $\lambda$ and $\alpha \in \mathbb Z_{\geq 0}^n$, we write $\SSYT_n(\lambda, \alpha)$ for the set of tableaux of shape $\lambda$ and content $\alpha$.
For a tableau $T$, $\sh(T)$ and $\wt(T)$ are its shape and content.
Its size $|T|$ counts its boxes, and $\width(T)$ counts its columns.
For a word or tableau $v$, $\wt_i(v)$ is the number of occurrences of $i$.
The Yamanouchi tableau $Y_\lambda$ has $\lambda_i$ copies of $i$ in row $i$.

The subtableau $T^{\leq j}$ consists of the entries at most $j$.
The row reading word $\rw(T)$ reads rows from bottom to top, each from left to right.
Column reading runs through the columns from left to right, each from bottom to top.
For a word or tableau without $1$'s, $\dec$ replaces each $i \geq 2$ by $i - 1$.

We write $T \leftarrow x$ for the tableau obtained by row inserting $x$ into $T$.
For a word $v$, insertion proceeds from left to right, and $P(v) = \emptyset \leftarrow v$ is its insertion tableau.
Two words $u, v$ are \emph{Knuth equivalent} if $P(u) = P(v)$~\cite{Knu1970}.

For a vertical or horizontal cut of $T$ into tableaux $A$ and $B$, we have
\[
	T = P(\rw(A)\rw(B)),
\]
where $A$ lies to the left of $B$ for a vertical cut and below $B$ for a horizontal cut.
Each of $\rw(A)$ and $\rw(B)$ may be replaced by the corresponding column reading word.

\subsection{Column complementation and the LS action}\label{subsec:complement}

For $k \geq 0$ and a partition $\mu \subseteq (k^n)$, we write $\overline\mu = (k - \mu_1, \ldots, k - \mu_n)$, so $\mu^\perp = w_0\overline\mu$.
For any tableau $T$ of width at most $k$, we append empty columns on the right to obtain $k$ columns, denoted by $C_1, \ldots, C_k$.
We define \emph{column complementation} $D_{n, k}$ by taking the $j$th column of $D_{n, k}T$ to be $[n] \setminus C_{k + 1 - j}$, written in increasing order~\cite{LasSch1981,Sta1999}.

\begin{prop}\label{prop:D}
Column complementation is an involution on tableaux of width at most $k$.
For partitions $\lambda, \mu \subseteq (k^n)$, column complementation gives a bijection
\[
	D_{n, k} : \SSYT_n(\lambda, \mu)
	\longrightarrow \SSYT_n(\lambda^\perp, \overline\mu).
\]
\end{prop}
\begin{proof}
By construction, $D_{n, k}^2T = T$, and $D_{n, k}T$ has shape $\lambda^\perp$ and content $\overline\mu$ when $T \in \SSYT_n(\lambda, \mu)$.
We show that $D_{n, k}T$ is semistandard.
Two adjacent columns $C, C'$, from left to right, have weakly increasing rows if and only if $|C \cap [r]| \geq |C' \cap [r]|$ for every $r \in [n]$.
The corresponding columns of $D_{n, k}T$ satisfy
\[
	\bigl|([n] \setminus C') \cap [r]\bigr|
	= r - |C' \cap [r]|
	\geq r - |C \cap [r]|
	= \bigl|([n] \setminus C) \cap [r]\bigr|.
\]
Thus $D_{n, k}T$ has weakly increasing rows and, by construction, strictly increasing columns, so it is semistandard.
\end{proof}

The Lascoux-Sch\"utzenberger (LS) action of $\mathfrak S_n$ on words is defined as follows~\cite[Section~4.3]{LasSch1981}.
For $s_i = (i, i + 1)$ with $1 \leq i < n$, we consider the subword consisting of $i$ and $i + 1$.
We repeatedly cancel adjacent pairs $(i + 1)i$ and replace the unmatched subword $i^a(i + 1)^b$ by $i^b(i + 1)^a$, keeping all other positions unchanged.
For a tableau $T$, we apply this rule to $\rw(T)$ and return the letters to their boxes to obtain the tableau $s_i \cdot T$.
Column reading gives the same action~\cite[Sections~3.1 and~6.1]{BumSch2017}.
Products act with the rightmost simple transposition first.
The operators satisfy the Coxeter relations and commute with insertion, i.e., $P(w \cdot v) = w \cdot P(v)$ for $w \in \mathfrak S_n$~\cite[Proposition~4.5]{LasSch1981}.
The longest element of $\mathfrak S_m$ is denoted by $w_0^{(m)}$, with $w_0 = w_0^{(n)}$.
We define $\Phi_{n, k}T = w_0 \cdot (D_{n, k}T)$.

\begin{lem}\label{lem:D-crystal}
For $1 \leq i < n$, we have $D_{n, k}s_i = s_iD_{n, k}$.
In particular, $\Phi_{n, k}^2 = \id$.
\end{lem}
\begin{proof}
In the column reading word, $i + 1$ and $i$ cancel within each column containing both.
It suffices to consider columns containing exactly one of $i, i + 1$.
Column complementation interchanges $i$ and $i + 1$ in these columns and reverses the column order.
Therefore, complementation preserves the cancelled pairs and sends the unmatched subword $i^a(i + 1)^b$ to $i^b(i + 1)^a$.
Thus $D_{n, k}s_iT = s_iD_{n, k}T$.
Since $D_{n, k}$ and $w_0$ are commuting involutions, $\Phi_{n, k}^2 = \id$.
\end{proof}

For partitions $\lambda, \mu \subseteq (k^n)$, we obtain a bijection
\[
	\Phi_{n, k} : \SSYT_n(\lambda, \mu) \longrightarrow
	\SSYT_n(\lambda^\perp, \mu^\perp).
\]

\begin{exmp}\label{ex:running-tableau}
For $n = 4$ and $k = 4$, the involution acts as follows:
\[
	T = \begin{ytableau}1&1&1&3\\2&2&4\\3&4\end{ytableau}
	\xrightarrow{D_{4, 4}}
	\begin{ytableau}1&2&3&4\\2&3\\4\end{ytableau}
	\xrightarrow{w_0^{(4)}}
	U = \begin{ytableau}1&1&3&3\\2&2\\4\end{ytableau}.
\]
Here $\sh(T) = (4, 3, 2, 0)$ and $\wt(T) = (3, 2, 2, 2)$, while
$\sh(U) = (4, 2, 1, 0)$ and $\wt(U) = (2, 2, 2, 1)$.
\end{exmp}

\subsection{Cyclage, cocharge, and charge}\label{subsec:cocharge}

For a partition $\mu$, a \emph{cyclage} of a tableau $T$ of content $\mu$ is a move
\[
	T = P(xv) \longrightarrow P(vx), \qquad x > 1,
\]
where $xv$ is any word inserting to $T$.
These moves define the \emph{cyclage graph} on tableaux of content $\mu$.
Every tableau $T$ admits a directed path to the one-row tableau, and the length of such a path depends only on $T$~\cite[Property~5.2]{LasSch1981}.
We define $\cocharge(T)$ to be this length.
For a word $w$ of content $\mu$, cocharge is given by $\cocharge(w) = \cocharge(P(w))$.
The empty word has $\cocharge(\emptyset) = 0$.
By definition,
\begin{equation}\label{eq:cocharge-rotation}
	\cocharge(xv) = \cocharge(vx) + 1
	\qquad\text{if $x > 1$ and $xv$ has content $\mu$}.
\end{equation}

\begin{exmp}
For content $(1, 1, 1)$, the tableaux in
\[
	\begin{ytableau}1\\2\\3\end{ytableau}
	\longrightarrow
	\begin{ytableau}1&3\\2\end{ytableau}
	\longrightarrow
	\begin{ytableau}1&2\\3\end{ytableau}
	\longrightarrow
	\begin{ytableau}1&2&3\end{ytableau}
\]
have cocharges $3, 2, 1, 0$, respectively.
\end{exmp}

For a partition $\mu$, we write $n(\mu) = \sum_{i \geq 1}(i - 1)\mu_i$.
For a word $v$ of content $\mu$, we define
\begin{equation*}
	\charge(v) = n(\mu) - \cocharge(v).
\end{equation*}
For a word $v$ of arbitrary content, we choose $\sigma \in \mathfrak S_n$ so that $\sigma \cdot v$ has partition content and define~\cite[Section~5]{LasSch1981}
\[
	\cocharge(v) = \cocharge(\sigma \cdot v), \qquad
	\charge(v) = \charge(\sigma \cdot v).
\]
The values $\cocharge(\sigma \cdot v)$ and $\charge(\sigma \cdot v)$ do not depend on the choice of $\sigma$.
For a tableau $T$, we set $\charge(T) = \charge(\rw(T))$ and $\cocharge(T) = \cocharge(\rw(T))$.

For partitions $\lambda, \mu$ with at most $n$ parts and $|\lambda| = |\mu|$, the charge formula is~\cite{LasSch1978}
\begin{equation}\label{eq:charge-formula}
	K_{\lambda, \mu}(q)
	= \sum_{T \in \SSYT_n(\lambda, \mu)}q^{\charge(T)}.
\end{equation}

Charge and cocharge are Knuth invariant and invariant under the LS action.
For a word with no $1$'s, $\dec$ agrees with the LS action of $s_{n - 1}\cdots s_1$ and therefore preserves cocharge.

\subsection{Charge preservation under \texorpdfstring{$\Phi_{n, k}$}{Phi(n,k)}}\label{subsec:charge-preservation}

For a tableau $T$, we write $a = \wt_1(T)$.
The tableau $T_L$ consists of the first $a$ columns with their first row $1^a$ removed.
The tableau $T_R$ consists of the remaining columns, shifted to the left.
We define the \emph{column catabolism} of $T$ by
\[
	\cat(T) = \dec\bigl(P(\rw(T_R)\rw(T_L))\bigr).
\]
This is the single-row case of column catabolism~\cite{KirShi2002}, followed by relabelling.
Note that column catabolism differs from the row katabolism defined in \cref{sec:katabolizability}.

\begin{rmk}\label{rem:cocharge-prefix}
The cyclage properties in~\cite[Properties~5.2 and~5.5]{LasSch1981} give
\begin{equation}\label{eq:cocharge-prefix}
	\cocharge(1^av) = \cocharge(v)
	\qquad(a \geq 0, \ \text{all letters of $v$ at least $2$}).
\end{equation}
\end{rmk}

\begin{lem}\label{lem:cocharge-catabolism}
For every tableau $T$,
\[
	\cocharge(T) = |T_L| + \cocharge(\cat(T)).
\]
\end{lem}
\begin{proof}
Let $w_L = \rw(T_L)$, $w_R = \rw(T_R)$, $W = w_L1^aw_R$, and $d = |T_L|$.
Cutting $T$ after column $a$ gives $P(W) = T$.
Cyclic rotation $\rot(xv) = vx$ commutes with the LS action~\cite[Theorem~4.6(2)]{LasSch1981}.
In other words, for any $\sigma \in \mathfrak S_n$,
\[
	\rot^d(\sigma \cdot W)
	= \sigma \cdot \rot^d(W)
	= \sigma \cdot (1^aw_Rw_L).
\]

We choose $\sigma$ so that $\sigma \cdot T$ has partition content.
Reading $\sigma \cdot T$ in the order used to form $W$ gives $\sigma \cdot W$~\cite[Sections~3.1 and~6.1]{BumSch2017}.
The first $d$ letters come from boxes below the first row, so they are greater than $1$.
Therefore, by \eqref{eq:cocharge-rotation} and \eqref{eq:cocharge-prefix},
\[
\begin{aligned}
	\cocharge(T) &= \cocharge(\sigma \cdot W)\\
	&= d + \cocharge\bigl(\rot^d(\sigma \cdot W)\bigr)\\
	&= d + \cocharge(1^aw_Rw_L)\\
	&= d + \cocharge(\cat(T)).
\end{aligned}
\]
Hence $\cocharge(T) = |T_L| + \cocharge(\cat(T))$.
\end{proof}

For tableaux $X, Y$, tableau multiplication is defined by $X \leftarrow Y = P(\rw(X)\rw(Y))$.
Then $\width(X \leftarrow Y) \leq \width(X) + \width(Y)$.
For $\width(X) \leq k_1$ and $\width(Y) \leq k_2$, column complementation reverses tableau multiplication~\cite[Property~3.4]{LasSch1981}:
\begin{equation}\label{eq:complement-product}
	D_{n, k_1 + k_2}(X \leftarrow Y)
	= D_{n, k_2}Y \leftarrow D_{n, k_1}X.
\end{equation}

\begin{lem}\label{lem:complement-catabolism}
If $n \geq 2$ and $\width(T) \leq k$, then
\[
	\cat(D_{n, k}T)
	= D_{n - 1, k}\bigl(\cat(T)\bigr).
\]
\end{lem}
\begin{proof}
Let $U = D_{n, k}T$, $X = \dec(T_L)$, and $Y = \dec(T_R)$.
The tableau $U$ has $k - a$ entries equal to $1$.
Complementation exchanges columns containing a $1$ with columns without a $1$ and reverses their order, so
\[
	\dec(U_L) = D_{n - 1, k - a}Y, \qquad
	\dec(U_R) = D_{n - 1, a}X.
\]
Since $\cat(T) = Y \leftarrow X$, \eqref{eq:complement-product} yields
\[
	\cat(U)
	= D_{n - 1, a}X \leftarrow D_{n - 1, k - a}Y
	= D_{n - 1, k}(Y \leftarrow X).
	\qedhere
\]
\end{proof}

\begin{prop}\label{prop:ambient-charge}
For every tableau $T$ with entries in $[n]$ and width at most $k$,
\[
	\charge(D_{n, k}T) = \charge(T), \qquad
	\charge(\Phi_{n, k}T) = \charge(T).
\]
\end{prop}
\begin{proof}
By \cref{lem:D-crystal} and the LS invariance of charge, we may assume that $T$ has partition content $\mu$.
We prove by induction on the alphabet size $n$ that
\begin{equation*}
	\cocharge(D_{n, k}T) - \cocharge(T)
	= k\binom n2 - (n - 1)|T|.
\end{equation*}
For $n = 1$, both tableaux have at most one row, so both cocharges are zero.
For $n \geq 2$, let $U = D_{n, k}T$.
The padded columns of $T$ have heights $h_1, \ldots, h_k$, and its first $a = \mu_1$ columns contain $1$.
Hence
\[
	|T_L| = \sum_{j = 1}^a(h_j - 1), \qquad
	|U_L| = \sum_{j = a + 1}^k(n - h_j - 1).
\]
Since $\sum_j h_j = |T|$, subtraction gives
$|U_L| - |T_L| = (n - 1)k - |T| - (n - 2)a$.
The tableau $\cat(T)$ has partition content $(\mu_2, \ldots, \mu_n)$ and $|T| - a$ boxes, with
\[
	\width(\cat(T)) \leq \width(T_L) + \width(T_R) \leq k.
\]
The induction hypothesis and \cref{lem:cocharge-catabolism,lem:complement-catabolism} give
\[
\begin{aligned}
	\cocharge(U) - \cocharge(T)
	&= (n - 1)k - |T| - (n - 2)a\\
	&\qquad + k\binom{n - 1}{2} - (n - 2)(|T| - a)\\
	&= k\binom n2 - (n - 1)|T|.
\end{aligned}
\]

By \cref{prop:D}, $D_{n, k}T$ has content $\overline\mu$, whose decreasing rearrangement is $\mu^\perp$.
Hence
\[
	n(\mu^\perp) - n(\mu)
	= k\binom n2 - (n - 1)|T|.
\]
Thus $\charge(D_{n, k}T) = \charge(T)$.
Hence the LS action of $w_0$ gives $\charge(\Phi_{n, k}T) = \charge(T)$.
\end{proof}

\section{Katabolizable tableaux and DARK crystals}\label{sec:katabolizability}

\subsection{Katabolizable tableaux}\label{subsec:catalan}

Let $\Psi \subseteq \Delta_n^+$ be an upper root ideal.
Taking boundaries gives a bijection between upper root ideals in $\Delta_n^+$ and Dyck paths.
The associated \emph{Hessenberg function} $h = (h_1, \ldots, h_n)$ is determined by
\[
	\Psi = \{(i, j) \in \Delta_n^+ : j > h_i\}.
\]
It satisfies $i \leq h_i \leq n$ and $h_1 \leq \cdots \leq h_n$.
Following~\cite[Section~2.1]{BlaMorPun2024}, we set $\dvec = \dvec(\Psi) = (d_1, \ldots, d_{n - 1})$, where
\[
	d_i = \bigl|\{j \in \{i, \ldots, n\} : (i, j) \notin \Psi\}\bigr|
		= h_i - i + 1\qquad(1 \leq i < n).
\]
The conditions on $h$ imply $1 \leq d_i \leq n - i + 1$ for $1 \leq i < n$ and $d_{i + 1} \geq d_i - 1$ for $1 \leq i < n - 1$.

\par\medskip
\noindent\begin{minipage}{\textwidth}
\begin{exmp}\label{ex:root-ideal-dyck}
For $n = 4$, let $\Psi = \Delta_4^+ \setminus \{(1, 2), (2, 3)\}$.
The boundary of $\Psi$ is the Dyck path shown in \cref{fig:root-dyck}.
The diagram gives $h = (2, 3, 3, 4)$ and $\dvec = (2, 2, 1)$.
The unshaded boxes $(1, 1), (1, 2)$ give $d_1 = 2$.
\begin{center}
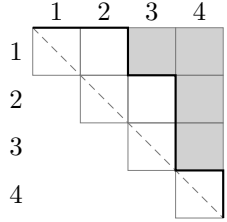

\begin{tikzpicture}[x=0.63cm,y=-0.63cm,font=\small]
	\foreach \i/\j in {1/3,1/4,2/4,3/4} {
		\fill[black!18] (\j-1,\i-1) rectangle (\j,\i);
	}
	\foreach \i in {1,...,4} {
		\foreach \j in {\i,...,4} {
			\draw[black!50,line width=0.3pt] (\j-1,\i-1) rectangle (\j,\i);
		}
	}
	\draw[densely dashed,black!55,line width=0.35pt] (0,0)--(4,4);
	\draw[line width=0.8pt] (0,0)--(2,0)--(2,1)--(3,1)--(3,3)--(4,3)--(4,4);
	\foreach \j in {1,...,4} {\node at (\j-0.5,-0.35) {$\j$};}
	\foreach \i in {1,...,4} {\node at (-0.35,\i-0.5) {$\i$};}
\end{tikzpicture}
\end{center}
\captionof{figure}{The shaded boxes form $\Psi$.}
\label{fig:root-dyck}
\end{exmp}
\end{minipage}
\par\medskip

A \emph{tabloid} is a list of weakly increasing rows, possibly empty, with no column condition.
A tabloid $A$ with $\ell$ fixed row positions is written $A = \langle A_1; \ldots; A_\ell\rangle$.
Its row reading word and insertion tableau are $\rw(A) = A_\ell\cdots A_1$ and $P(A) = P(\rw(A))$, respectively.
For $1 \leq r < \ell$, we define the \emph{partial insertion} $P_r$ by replacing rows $r$ and $r + 1$ of $A$ with the first and second rows of $P(A_{r + 1}A_r)$, respectively~\cite[Definition~6.8]{BlaMorPun2024}.
All other rows remain unchanged, and missing rows are replaced by empty rows.

The row $u_r$ is obtained from $A_r$ by deleting its $1$'s and decreasing the remaining entries by one.
For $1 \leq d \leq \ell$, the katabolism operators are defined by~\cite[Definitions~6.8 and~6.14]{BlaMorPun2024}
\[
	\kat(A) = \langle u_2; \ldots; u_\ell; u_1\rangle, \qquad
	F_{d, \ell} = P_d \circ \cdots \circ P_{\ell - 1} \circ \kat.
\]
For $d = \ell$, there is no insertion.
\begin{defn}[{\cite[Definition~2.15 and Proposition~6.15]{BlaMorPun2024}}]\label{def:katabolizable}
For an upper root ideal $\Psi \subseteq \Delta_n^+$ and a tableau $T$ with entries in $[n]$, we define its katabolism sequence by
\begin{equation}\label{eq:katabolism-sequence}
	A^{(0)} = T, \qquad A^{(i)} = F_{d_i, n}(A^{(i - 1)})\quad(1 \leq i < n),
\end{equation}
where $T$ is viewed as a tabloid with $n$ row positions.
We say that $T$ is \emph{$\Psi$-katabolizable} if, for every $0 \leq i < n$, all $1$'s in $A^{(i)}$ lie in its first row.
For a partition $\lambda$ and $\alpha \in \mathbb Z_{\geq 0}^n$, the set of $\Psi$-katabolizable tableaux of shape $\lambda$ and content $\alpha$ is denoted by $\tset_\Psi(\lambda, \alpha)$.
\end{defn}

\begin{exmp}\label{ex:katabolism}
For $T$ from \cref{ex:running-tableau} and $\Psi$ from \cref{ex:root-ideal-dyck}, we have $\dvec = (2, 2, 1)$.
As a four-row tabloid, $A^{(0)} = T = \langle1113; 224; 34; \emptyset\rangle$.
We delete the $1$'s and decrease the other entries, obtaining the rows $2, 113, 23, \emptyset$.
We move the first row to the bottom to obtain $\kat(T) = \langle113; 23; \emptyset; 2\rangle$.
Since $d_1 = 2$, we apply $P_3$ and then $P_2$:
\[
	\langle113; 23; \emptyset; 2\rangle
	\xrightarrow{P_3} \langle113; 23; 2; \emptyset\rangle
	\xrightarrow{P_2} \langle113; 223; \emptyset; \emptyset\rangle = A^{(1)}.
\]
The first two rows of $A^{(1)}$ both have $3$ in their third column, so $A^{(1)}$ is not semistandard.
The remaining steps use $d_2 = 2$ and $d_3 = 1$:
\[
	A^{(1)} \xrightarrow{F_{2, 4}}
	\langle112; 2; \emptyset; \emptyset\rangle = A^{(2)}
	\xrightarrow{F_{1, 4}} \langle11; \emptyset; \emptyset; \emptyset\rangle = A^{(3)}.
\]
All $1$'s in $A^{(i)}$ lie in its first row for $0 \leq i \leq 3$, so $T$ is $\Psi$-katabolizable.

For $U = \Phi_{4, 4}T$ from \cref{ex:running-tableau}, we have $\Psi^\dagger = \Delta_4^+ \setminus \{(2, 3), (3, 4)\}$ and $\dvec(\Psi^\dagger) = (1, 2, 2)$.
Its katabolism sequence is
\[
	\langle1133; 22; 4; \emptyset\rangle
	\xrightarrow{F_{1, 4}} \langle113; 22; \emptyset; \emptyset\rangle
	\xrightarrow{F_{2, 4}} \langle11; 2; \emptyset; \emptyset\rangle
	\xrightarrow{F_{2, 4}} \langle1; \emptyset; \emptyset; \emptyset\rangle.
\]
All $1$'s remain in the first row, so $U$ is $\Psi^\dagger$-katabolizable.
\end{exmp}

\begin{thm}[{\cite[Theorem~2.18]{BlaMorPun2024}}]\label{thm:BMP}
For an upper root ideal $\Psi \subseteq \Delta_n^+$ and partitions $\lambda, \mu$ with at most $n$ parts,
\[
	K^\Psi_{\lambda, \mu}(q) = \sum_{T \in \tset_\Psi(\lambda, \mu)}q^{\charge(T)}.
\]
\end{thm}
For $\Psi = \Delta_n^+$, we have $d_i = 1$ for all $1 \leq i < n$, so every step uses $F_{1, n}$ and produces a tableau.
Thus every tableau is $\Psi$-katabolizable, and the formula reduces to \eqref{eq:charge-formula}.

\subsection{DARK crystals}\label{subsec:crystal-prelim}

Kashiwara and Lusztig independently introduced crystal theory to study representations of quantum groups combinatorially~\cite{Kas1991,Lus1990}.
We follow~\cite{BumSch2017,HonKan2002} for the basic definitions and notation.

For a crystal $B$ with index set $I$, we denote its crystal operators by $\widetilde e_i, \widetilde f_i$ and its weight function by $\wt$.
If an operator cannot be applied, its value is $0$.
For $b \in B$ and $i \in I$, we set
\[
	\varepsilon_i(b) = \max\{a \in \mathbb Z_{\geq 0} : \widetilde e_i^{\,a}b \neq 0\},
	\qquad
	\varphi_i(b) = \max\{a \in \mathbb Z_{\geq 0} : \widetilde f_i^{\,a}b \neq 0\}.
\]
The crystal graph has an arrow $b \xrightarrow{i} b'$ when $\widetilde f_i b = b'$, or equivalently $\widetilde e_i b' = b$.

For crystals $B$ and $C$ of the same type, the tensor product $B \otimes C$ has underlying set $B \times C$ and weight $\wt(x \otimes y) = \wt(x) + \wt(y)$.
We use the tensor product convention opposite to Kashiwara's~\cite[Section~2.4]{Kas1991}:
\[
\begin{aligned}
	\widetilde e_i(x \otimes y)&=
	\begin{cases}
		\widetilde e_i x \otimes y, &\varepsilon_i(x) > \varphi_i(y), \\
		x \otimes \widetilde e_i y, &\varepsilon_i(x) \leq \varphi_i(y),
	\end{cases}\\
	\widetilde f_i(x \otimes y)&=
	\begin{cases}
		\widetilde f_i x \otimes y, &\varepsilon_i(x) \geq \varphi_i(y), \\
		x \otimes \widetilde f_i y, &\varepsilon_i(x) < \varphi_i(y).
	\end{cases}
\end{aligned}
\]
A tensor product containing $0$ is interpreted as $0$.
The functions $\varepsilon_i$ and $\varphi_i$ satisfy
\begin{align}
	\varepsilon_i(x \otimes y)
		&= \varepsilon_i(y) + \max(0, \varepsilon_i(x) - \varphi_i(y)),
		\label{eq:tensor-epsilon}\\
	\varphi_i(x \otimes y)
		&= \varphi_i(x) + \max(0, \varphi_i(y) - \varepsilon_i(x)).
		\label{eq:tensor-phi}
\end{align}

We use the single-row Kirillov-Reshetikhin (KR) crystals $B^{1, s}$ of affine type $A_{n - 1}^{(1)}$, with $n \geq 2$ and index set $I = \{0, 1, \ldots, n - 1\}$~\cite[Sections~2.7 and~4.2]{SchTin2012}.
For $s \geq 1$, the elements of $B^{1, s}$ are weakly increasing rows $x = x_1\cdots x_s$ of length $|x| = s$, with $x_i \in [n]$.
For $1 \leq i < n$, the operator $\widetilde f_i$ replaces one $i$ by $i + 1$, while $\widetilde f_0$ replaces one $n$ by $1$.
The resulting row is sorted, and $\widetilde e_i$ reverses the corresponding operation.
We allow empty rows by setting $B^{1, 0} = \{\emptyset\}$.
For $\alpha = (\alpha_1, \ldots, \alpha_p) \in \mathbb Z_{\geq 0}^p$, we set $B^\alpha = B^{1, \alpha_p} \otimes \cdots \otimes B^{1, \alpha_1}$ and write its elements as
\[
	b = b_p \otimes \cdots \otimes b_1,
	\qquad b_j \in B^{1, \alpha_j}\quad(1 \leq j \leq p).
\]
We call each $b_j$ a \emph{tensor factor} of $b$, and its entries \emph{letters}.
The $\mathfrak{gl}_n$ weight $\wt(b)$ is its content vector.
For $1 \leq i < n$, we have $\varphi_i(b) - \varepsilon_i(b) = \langle\wt(b), \evec_i - \evec_{i + 1}\rangle$.
The connected components after deleting the $0$-arrows are called \emph{classical components}.
An element $b$ is \emph{classically highest} if $\widetilde e_i b = 0$ for all $1 \leq i < n$, and \emph{classically lowest} if $\widetilde f_i b = 0$ for all $1 \leq i < n$.

Row reading embeds tableaux of shape $\lambda$ with entries in $[n]$ into $(B^{1, 1})^{\otimes |\lambda|}$ as a classical subcrystal~\cite[Section~3.2]{Shi2002Affine}.
The LS action of $s_i$ sends each tableau to the tableau on the same $i$-string with $\varepsilon_i$ and $\varphi_i$ interchanged.

For $1 \leq i < n$, we use $\F_i(X) = \{\widetilde f_i^{\,a}b : b \in X, \ a \geq 0\} \setminus \{0\}$.
For a reduced expression $w = s_{i_1}\cdots s_{i_r}$, we set $\F_w = \F_{i_1} \circ \cdots \circ \F_{i_r}$.
The same notation applies to row crystals with any alphabet $[\ell]$, where $\ell \geq 2$.
For $b \in B^\alpha$ with letters in $[\ell]$, the map $\tau_\ell$ replaces each $j < \ell$ by $j + 1$ and each $\ell$ by $1$, then sorts the letters in each tensor factor.
We write $\tau = \tau_\ell$ when $\ell$ is clear and set $\tau(0) = 0$.
For $i \in \mathbb Z/\ell\mathbb Z$, we have~\cite[Proposition~5.5 and (5.1)]{Nao2013}
\begin{equation}\label{eq:tau-epsilon}
	\tau\widetilde e_i = \widetilde e_{i + 1}\tau, \qquad
	\tau\widetilde f_i = \widetilde f_{i + 1}\tau, \qquad
	\varepsilon_i(\tau b) = \varepsilon_{i - 1}(b).
\end{equation}

We recall the DARK (Kirillov-Reshetikhin affine Demazure) construction in the form used here~\cite[Definition~2.10]{BlaMorPun2024}.
\begin{defn}\label{def:DARK}
For $1 \leq p \leq \ell$ with $\ell \geq 2$, let $\Psi \subseteq \Delta_p^+$ be an upper root ideal and $\mu = (\mu_1, \ldots, \mu_p)$ a partition.
With alphabet $[\ell]$ and $\tau = \tau_\ell$, we define $\mathcal D_{\Psi, \mu} \subseteq B^\mu$ for $p \geq 2$ by
\begin{equation}\label{eq:root-tail-stages}
\begin{aligned}
	X_p^\Psi(\mu)&= \F_{s_{\ell - 1}\cdots s_{d_{p - 1}}}\{1^{\mu_p}\}, \\
	X_j^\Psi(\mu)&= \F_{s_{\ell - 1}\cdots s_{d_{j - 1}}}
		\bigl(\tau(X_{j + 1}^\Psi(\mu)) \otimes \{1^{\mu_j}\}\bigr)
		\quad(2 \leq j < p), \\
	Z_{\Psi, \mu}&= \tau(X_2^\Psi(\mu)) \otimes \{1^{\mu_1}\},
	\qquad \mathcal D_{\Psi, \mu} = \F_{w_0^{(\ell)}}Z_{\Psi, \mu},
\end{aligned}
\end{equation}
where $\dvec(\Psi) = (d_1, \ldots, d_{p - 1})$ and $\F_{s_{\ell - 1}\cdots s_d}$ is the identity when $d = \ell$.
For $p = 1$, we set $Z_{\emptyset, \mu} = \{1^{\mu_1}\}$ and $\mathcal D_{\emptyset, \mu} = \F_{w_0^{(\ell)}}Z_{\emptyset, \mu}$.
\end{defn}
We write $\mathcal D_{\Psi, \mu}^{[\ell]}$ when the alphabet needs to be specified and omit the superscript when it is clear.
The recursively defined subsets are independent of the reduced expressions~\cite[Theorem~3.7 and Corollary~3.11]{Bla2020}.

\begin{exmp}\label{ex:root-recursion}
For $p = \ell = 3$, we restrict the ideal in \cref{ex:root-ideal-dyck} to its last three vertices and relabel by $[3]$, giving $\Psi = \{(1, 3), (2, 3)\}$.
We use $\mu = (1, 1, 1)$, with $\dvec = (2, 1)$.
We have
\[
\begin{aligned}
	X_3^\Psi(\mu)&= \F_2\F_1\{1\} = \{1, 2, 3\}, \\
	X_2^\Psi(\mu)&= \F_2\{1 \otimes 1, 2 \otimes 1, 3 \otimes 1\}\\
		&= \{1 \otimes 1, 2 \otimes 1, 3 \otimes 1\}, \\
	Z_{\Psi, \mu}&= \{1 \otimes 2 \otimes 1, 2 \otimes 2 \otimes 1, 3 \otimes 2 \otimes 1\}.
\end{aligned}
\]
For $w_0 = s_1s_2s_1$, the composition $\F_1 \circ \F_2 \circ \F_1$ gives
\[
	\mathcal D_{\Psi, \mu} = \left\{
	\begin{gathered}
		1 \otimes 2 \otimes 1, \quad1 \otimes 3 \otimes 1, \quad1 \otimes 3 \otimes 2, \\
		2 \otimes 2 \otimes 1, \quad2 \otimes 3 \otimes 1, \quad2 \otimes 3 \otimes 2, \\
		3 \otimes 2 \otimes 1, \quad3 \otimes 3 \otimes 1, \quad3 \otimes 3 \otimes 2
	\end{gathered}\right\}.
\]
The classical highest-weight elements are $3 \otimes 2 \otimes 1$ and $1 \otimes 2 \otimes 1$.
\end{exmp}

For a tabloid $A = \langle A_1; \ldots; A_n\rangle$ of content $\alpha = (\alpha_1, \ldots, \alpha_p)$, we set
\[
\begin{aligned}
	b_j(A)&= 1^{\wt_j(A_1)}\cdots n^{\wt_j(A_n)} \in B^{1, \alpha_j}
	\qquad(1 \leq j \leq p), \\
	\inv(A)&= b_p(A) \otimes \cdots \otimes b_1(A).
\end{aligned}
\]
The tensor factor $b_j(A)$ contains one letter $i$ for each occurrence of $j$ in row $A_i$.

\begin{exmp}\label{ex:row-encoding}
For $T$ from \cref{ex:running-tableau}, the tensor factors are $b_1(T) = 111$, $b_2(T) = 22$, $b_3(T) = 13$, and $b_4(T) = 23$.
Thus $\inv(T) = (23) \otimes (13) \otimes (22) \otimes (111)$.
\end{exmp}

For $b \in B^\alpha$ with letters in $[\ell]$, we start with $P = Q = \emptyset$ and process the tensor factors in the order $b_1, b_2, \ldots, b_p$~\cite[Appendix~A.4.1, Proposition~2]{Ful1997}.
Within each tensor factor $b_j$, the letters are inserted into $P$ from right to left by column insertion.
Whenever a new box is created in $P$, the entry $j$ is placed in the corresponding box of $Q$.
The resulting insertion and recording tableaux are denoted $P(b)$ and $Q(b)$.
They have the same shape, and $Q(b)$ has content $\alpha$.
Classical crystal operators act on $P(b)$ and fix $Q(b)$~\cite[Proposition~3.3]{Shi2002Affine}.
The indices $j$ increase in insertion order, so
\begin{equation}\label{eq:recording-prefix}
	Q(b)^{\leq m} = Q(b_m \otimes \cdots \otimes b_1)\qquad(1 \leq m \leq p).
\end{equation}
For a tableau $T$ with at most $\ell$ rows and content $\alpha$, $\mathcal C_T = \{b \in B^\alpha : Q(b) = T\}$
is a classical component with highest element $\inv(T)$~\cite[Proposition~3.4 and Remark~3.7]{Shi2002Affine}.
If $T$ has shape $\lambda$, then $P(\inv(T))$ is the Yamanouchi tableau $Y_\lambda$.
For any tabloid $A$, we have~\cite[Appendix~A.4.1]{Ful1997}
\begin{equation}\label{eq:RSK-tabloid}
	Q(\inv(A)) = P(\rw(A)).
\end{equation}

For $\Psi \subseteq \Delta_p^+$ and $\ell \geq p$, appending $\ell - p$ empty rows below row $p$ does not change $\Psi$-katabolizability.
After each application of $\kat$, the operators $P_{\ell - 1}, \ldots, P_p$ move the last row to row $p$, and the remaining partial insertions give the $p$-row computation.

\begin{prop}[{\cite[Theorem~6.21 and Proposition~6.15]{BlaMorPun2024}}]\label{prop:BMP-components}
For an upper root ideal $\Psi \subseteq \Delta_p^+$, a partition $\mu = (\mu_1, \ldots, \mu_p)$, and $\ell \geq \max\{p, 2\}$, the set $\mathcal D_{\Psi, \mu}$ with alphabet $[\ell]$ decomposes into classical components:
\begin{equation*}
	\mathcal D_{\Psi, \mu}
	= \bigsqcup_{\substack{T\text{ of content }\mu\\T\text{ is }\Psi\text{-katabolizable}}}\mathcal C_T.
\end{equation*}
\end{prop}
For a partition $\lambda$, the map $T \mapsto \inv(T)$ gives a bijection from $\tset_\Psi(\lambda, \mu)$ to the classical highest-weight elements of $\mathcal D_{\Psi, \mu}$ of weight $\lambda$.
Its inverse is $b \mapsto Q(b)$.
For example, in \cref{ex:root-recursion}, $Q(3 \otimes 2 \otimes 1) = \begin{ytableau}1\\2\\3\end{ytableau}$ and $Q(1 \otimes 2 \otimes 1) = \begin{ytableau}1&3\\2\end{ytableau}$.

\section{Preservation of katabolizability under \texorpdfstring{$\Phi_{n, k}$}{Phi(n,k)}}\label{sec:dark-proof}

For the inductive step with strictly decreasing content, we characterize $\Psi$-katabolizability using tableaux of contents $(\mu_1, \ldots, \mu_{n - 1})$ and $(\mu_2, \ldots, \mu_n)$.
The map $\Phi_{n, k}$ interchanges the roles of the two tableaux, giving an induction on $n$.
We then extend the result to arbitrary partition content by a piecewise-linear argument.

\subsection{Tableau restrictions and katabolizability}\label{subsec:traces}

We first compare the katabolism sequences of $T$ and $T^{\leq n - 1}$, whose content is $(\mu_1, \ldots, \mu_{n - 1})$.

We use tabloids with $\ell \geq 2$ row positions and row crystals with alphabet $[\ell]$.
We write $\widetilde e_i^{\max}(b) = \widetilde e_i^{\,\varepsilon_i(b)}b$.
Products of operators act from right to left, and an empty product is the identity.
For a tabloid $A$, the partial insertion and katabolism operators satisfy~\cite[Propositions~6.9 and~6.13]{BlaMorPun2024}
\begin{equation}\label{eq:kat-crystal}
	\inv(P_iA) = \widetilde e_i^{\max}\inv(A),
	\qquad
	\inv(F_{d, \ell}A)
	= \widetilde e_d^{\max}\cdots\widetilde e_{\ell - 1}^{\max}
		\tau_\ell^{-1}\pi\inv(A),
\end{equation}
where $\pi$ removes the rightmost tensor factor.
We say that $A$ satisfies the \emph{first-row condition} if all $1$'s of $A$ lie in its first row.
This holds exactly when $b_1(A) = 1^{|b_1(A)|}$.

\begin{exmp}\label{ex:crystal-step}
For $T$ and $A^{(1)}$ from \cref{ex:katabolism}, we have $\inv(T) = (23) \otimes (13) \otimes (22) \otimes (111)$.
Thus $\tau_4^{-1}\pi\inv(T) = (12) \otimes (24) \otimes (11)$.
Applying $\widetilde e_3^{\max}$ and then $\widetilde e_2^{\max}$ gives
\[
	(12) \otimes (24) \otimes (11)
	\xrightarrow{\widetilde e_3^{\max}} (12) \otimes (23) \otimes (11)
	\xrightarrow{\widetilde e_2^{\max}} (12) \otimes (22) \otimes (11) = \inv(A^{(1)}).
\]
Thus the second identity in \eqref{eq:kat-crystal} holds in this example.
\end{exmp}

For an arbitrary $\ell$-row tabloid $A^{(0)}$ and indices $d_t \in [\ell]$, we define $A^{(t)}$ by iterating the recurrence in \eqref{eq:katabolism-sequence} with $n = \ell$.
For a word $v$, the subword $v^{\geq j}$ consists of all letters at least $j$.

\begin{lem}\label{lem:top-rows}
For $t \geq 0$ and $2 \leq j \leq \ell$ satisfying $d_{t + a} \geq j - a + 1$ for $1 \leq a < j$, the tabloid $A^{(t + j - 1)}$ satisfies the first-row condition if and only if every entry $j$ of $A^{(t)}$ lies in row $j$.
\end{lem}
\begin{proof}
For $1 \leq a < j$, the partial insertions in $F_{d_{t + a}, \ell}$ leave row $j - a$ unchanged since $d_{t + a} \geq j - a + 1$.
Hence
\[
	A^{(t + a)}_{j - a}
	= (\kat A^{(t + a - 1)})_{j - a}
	= \dec\bigl((A^{(t + a - 1)}_{j - a + 1})^{\geq 2}\bigr).
\]
Applying this identity for $a = j - 1, j - 2, \ldots, 1$ with $(\dec v)^{\geq c} = \dec(v^{\geq c + 1})$, we obtain
\begin{equation*}
	A^{(t + j - 1)}_1 = \dec^{\,j - 1}\bigl((A^{(t)}_j)^{\geq j}\bigr).
\end{equation*}
Since $\kat$ deletes the $1$'s and decreases the other entries by one, the $1$'s in $A^{(t + j - 1)}$ come exactly from the entries $j$ of $A^{(t)}$.
Therefore, all these $1$'s lie in the first row if and only if every entry $j$ of $A^{(t)}$ lies in row $j$.
\end{proof}

For a partition $\mu$ with at most $\ell$ parts, \cref{lem:top-rows} with $t = 0$ and $\dvec(\emptyset) = (\ell, \ell - 1, \ldots, 2)$ shows that $Y_\mu$ is the only $\emptyset$-katabolizable tableau of content $\mu$.

\begin{lem}\label{lem:prefix-reduction}
For a tableau $T$ with entries in $[\ell]$ and an upper root ideal $\Psi \subseteq \Delta_\ell^+$, let $A^{(0)}, \ldots, A^{(\ell - 1)}$ be the $\Psi$-katabolism sequence of $T$.
For $2 \leq m \leq \ell$, the tableau $T^{\leq m}$ is $(\Psi \cap \Delta_m^+)$-katabolizable if and only if $A^{(1)}, \ldots, A^{(m - 1)}$ satisfy the first-row condition.
\end{lem}
\begin{proof}
For $\dvec(\Psi) = (d_1, \ldots, d_{\ell - 1})$, we have
$\dvec(\Psi \cap \Delta_m^+) = (\bar d_1, \ldots, \bar d_{m - 1})$,
where $\bar d_i = \min\{d_i, m - i + 1\}$.
Let $\bar A^{(0)}, \ldots, \bar A^{(m - 1)}$ be the
$(\Psi \cap \Delta_m^+)$-katabolism sequence of $T^{\leq m}$.
We claim that, for every $1 \leq t < m$, $A^{(t)}$ satisfies the first-row condition if and only if $\bar A^{(t)}$ does.

By \eqref{eq:kat-crystal}, the rightmost $q - 1$ tensor factors of
$\inv(F_{d, \ell}A)$ depend only on $d$ and the rightmost $q$ tensor
factors of $\inv(A)$.
The tableau $T^{\leq m}$ lies in the first $m$ rows, and
\[
	\inv(T^{\leq m}) = b_m(T) \otimes \cdots \otimes b_1(T).
\]
For $b \in B^\alpha$ whose letters lie in $[m]$, we have
\[
	\widetilde e_m^{\max}\cdots\widetilde e_{\ell - 1}^{\max}
	\tau_\ell^{-1}b = \tau_m^{-1}b.
\]
For $0 \leq t < m$ with $d_u \leq m - u + 1$ for $1 \leq u \leq t$,
\eqref{eq:kat-crystal} and the formula for $\tau_m^{-1}b$ give, by induction,
\[
	\inv(\bar A^{(t)})
	= b_{m - t}(A^{(t)}) \otimes \cdots \otimes b_1(A^{(t)}).
\]
Hence $A^{(t)}$ and $\bar A^{(t)}$ satisfy the first-row condition
simultaneously.
If $d_u \leq m - u + 1$ for every $1 \leq u < m$, this proves the claim.

Otherwise, we take the least index $1 \leq i < m$ with $d_i > m - i + 1$ and set $r = m - i + 1 \geq 2$.
Then
\begin{equation*}
	\inv(\bar A^{(t)})
	= b_{m - t}(A^{(t)}) \otimes \cdots \otimes b_1(A^{(t)})
	\qquad(0 \leq t < i).
\end{equation*}
For $0 \leq a \leq r - 2$, we have
\[
	d_{i + a} \geq d_i - a \geq r - a + 1, \qquad
	\bar d_{i + a}
	= \min\{d_{i + a}, m - i - a + 1\} = r - a.
\]
At $t = i - 1$, we have
$b_j(A^{(i - 1)}) = b_j(\bar A^{(i - 1)})$ for $2 \leq j \leq r$.

For $2 \leq j \leq r$ and $1 \leq a < j$, we have
$d_{i + a - 1} \geq \bar d_{i + a - 1} = r - a + 1 \geq j - a + 1$.
Thus \cref{lem:top-rows} applies to both sequences with $t = i - 1$,
and the first-row conditions for $A^{(i + j - 2)}$ and
$\bar A^{(i + j - 2)}$ are respectively equivalent to
\[
	b_j(A^{(i - 1)}) = j^{|b_j(A^{(i - 1)})|},
	\qquad
	b_j(\bar A^{(i - 1)}) = j^{|b_j(\bar A^{(i - 1)})|}.
\]

Therefore, $T^{\leq m}$ is $(\Psi \cap \Delta_m^+)$-katabolizable
if and only if $A^{(1)}, \ldots, A^{(m - 1)}$ satisfy the
first-row condition.
\end{proof}

\subsection{The combinatorial \texorpdfstring{$R$}{R}-matrix}\label{subsec:strict-carrier}

We establish the $R$-matrix properties needed to study katabolizability for content $(\mu_2, \ldots, \mu_n)$.

For $s, t \geq 0$, the combinatorial $R$-matrix is an affine crystal isomorphism~\cite{KanKasMisMiwNakNak1992}
\[
	R : B^{1, s} \otimes B^{1, t} \longrightarrow B^{1, t} \otimes B^{1, s}.
\]

For a reduced expression $w = s_{i_1}\cdots s_{i_r} \in \mathfrak S_p$, we set
\[
	R_w = R_{s_{i_1}} \circ \cdots \circ R_{s_{i_r}},
\]
where $R_{s_i} : B^\alpha \longrightarrow B^{s_i\alpha}$ acts as $R$ on $B^{1, \alpha_{i + 1}} \otimes B^{1, \alpha_i}$ and as the identity on the other tensor factors.
The map $R_w$ is independent of the reduced expression and satisfies~\cite[(3.11)-(3.12) and Remark~3.20]{Shi2002Affine}
\begin{equation}\label{eq:R-recording}
	Q(R_wb) = w \cdot Q(b),
\end{equation}
where $w$ acts on $Q(b)$ by the LS action.

For $b \in B^\alpha$ and $y \in B^{1, s}$, we write
\[
	R(b \otimes y) = y' \otimes v,
	\qquad y' \in B^{1, s}, \quad v \in B^\alpha,
\]
where $R$ moves $y$ to the leftmost position by acting on adjacent tensor factors from right to left.

For a tensor product of row crystals $\mathcal B$, a subset $Y \subseteq \mathcal B$, and $a \geq 0$, we consider the following condition.
\begin{equation}\label{eq:zero-string-condition}
\begin{gathered}
	b \in Y\quad\text{if and only if}\quad\widetilde f_0b \in Y\\
	\text{for every }b \in \mathcal B\text{ with }
	\varepsilon_0(b) > a\text{ and }\widetilde f_0b \neq 0.
\end{gathered}
\end{equation}

For the highest element $u_{L\Lambda_0}$ of the affine highest-weight crystal $B(L\Lambda_0)$, we have
\[
	\varepsilon_i(u_{L\Lambda_0}) = 0,
	\qquad\varphi_i(u_{L\Lambda_0}) = L\delta_{i0}
	\qquad(L \geq 0, \ 0 \leq i < \ell),
\]
where $\Lambda_0$ is the fundamental weight corresponding to the affine node $0$.
A $0$-string $S$ is a maximal chain of $0$-arrows.
Its highest vertex $u_S$ satisfies $\widetilde e_0u_S = 0$.

\begin{lem}\label{lem:dark-zero}
We assume that $\mu = (\mu_1, \ldots, \mu_p)$ is a partition with $2 \leq p \leq \ell$ and $\Psi \subseteq \Delta_p^+$ is an upper root ideal.
The set $\mathcal D_{\Psi, \mu}$ satisfies \eqref{eq:zero-string-condition} for every $a \geq \mu_1$.
\end{lem}
\begin{proof}
For $b \in B^\mu$ with $\varepsilon_0(b) > a \geq \mu_1$ and $\widetilde f_0b \neq 0$, we have
\[
	\varepsilon_0(b \otimes u_{\mu_1\Lambda_0}) = \varepsilon_0(b) - \mu_1 > 0,
	\qquad
	\widetilde f_0(b \otimes u_{\mu_1\Lambda_0}) = \widetilde f_0b \otimes u_{\mu_1\Lambda_0} \neq 0.
\]
By~\cite[Theorem~2.2 and Corollary~3.11]{Bla2020}, $\mathcal D_{\Psi, \mu} \otimes \{u_{\mu_1\Lambda_0}\}$ is a disjoint union of Demazure crystals in $B^\mu \otimes B(\mu_1\Lambda_0)$.
The intersection of a Demazure crystal with a $0$-string $S$ is $\emptyset$, $\{u_S\}$, or $S$~\cite[Proposition~3.3.5]{Kas1993}.
Therefore, for every $0$-string $S$ in $B^\mu \otimes B(\mu_1\Lambda_0)$, we have
\begin{equation*}
	\bigl(\mathcal D_{\Psi, \mu} \otimes \{u_{\mu_1\Lambda_0}\}\bigr) \cap S
	\in \bigl\{\emptyset, \{u_S\}, S\bigr\}.
\end{equation*}
Both tensors lie in $S \setminus \{u_S\}$ for the same $0$-string $S$, so \eqref{eq:zero-string-condition} holds for $a$.
\end{proof}

For a partition $\mu = (\mu_1, \ldots, \mu_p)$ with $1 \leq p < \ell$, let $\tau = \tau_\ell$ and $B^\mu = B^{1, \mu_p} \otimes \cdots \otimes B^{1, \mu_1}$.
For $b \in B^\mu$ and $L > \mu_1$, let
\begin{equation}\label{eq:carrier-body}
	R(\tau b \otimes 1^L) = y' \otimes v(b),
	\qquad y' \in B^{1, L}, \quad v(b) \in B^\mu.
\end{equation}

\begin{prop}\label{prop:strict-carrier}
We assume that $Y \subseteq B^\mu$ is a union of classical components and satisfies \eqref{eq:zero-string-condition} with $a = \mu_1$.
Then $b \in Y$ if and only if $v(b) \in Y$.
\end{prop}
\begin{proof}
We may assume that $\mu_p > 0$.
We first prove that $v(b) = b$ when $b$ is classically highest and $\varepsilon_0(b) \leq L$.
For $b = x_p \otimes \cdots \otimes x_1$, we proceed by induction on $j$, where $1 \leq j \leq p$.
Let $x = x_j$ and $z = x_{j - 1} \otimes \cdots \otimes x_1$.
We assume inductively that $R(\tau z \otimes 1^L) = y \otimes z$ with $y \in B^{1, L}$.
For $j = 1$, $z$ is the empty tensor product and $y = 1^L$.

We have $\wt(y) = L\evec_1 + \wt(\tau z) - \wt(z)$.
Since $x \otimes z$ is classically highest, we have $\varepsilon_{i - 1}(z) = 0$ and $\wt_i(x) = \varepsilon_{i - 1}(x) \leq \varphi_{i - 1}(z)$ for $2 \leq i \leq \ell$.
Since $\wt_i(b) = 0$ for $p < i \leq \ell$, $\tau$ adds $1$ to every letter of $b$.
By \eqref{eq:tau-epsilon},
\[
	\wt_1(b) = \wt_2(\tau b) = \varepsilon_1(\tau b) = \varepsilon_0(b) \leq L.
\]
Together with $\wt_1(x) + \wt_1(z) \leq \wt_1(b)$, we obtain 
\[
\begin{aligned}
	\wt_i(x) &\leq \wt_{i - 1}(z) - \wt_i(z) = \wt_i(y)
		\qquad(2 \leq i \leq \ell), \\
	\wt_1(x) &\leq L - \wt_1(z) = \wt_1(y).
\end{aligned}
\]
Since $|\tau x| \leq \mu_1 < L = |y|$ and $\wt_{t - 1}(y) \geq \wt_{t - 1}(x) = \wt_t(\tau x)$, we pair each letter $t$ of $\tau x$, from largest to smallest, with an unused $t - 1$ in $y$.
By the pairing rule~\cite[Rule~3.11 and Proposition~3.21]{NakYam1997}, we obtain $R(\tau x \otimes y) = y' \otimes x$.
Therefore,
\[
	R(\tau(x \otimes z) \otimes 1^L)
	= R(\tau x \otimes y) \otimes z
	= y' \otimes x \otimes z.
\]
When $j = p$, we obtain $v(b) = b$.

Now let $b \in B^\mu$ be arbitrary.
The connected component of $b \otimes u_{L\Lambda_0}$ has affine highest element $\bar b \otimes u_{L\Lambda_0}$~\cite[Proposition~8.1]{SchTin2012}.
Then $\bar b$ is classically highest and $\varepsilon_0(\bar b) \leq L$, so $v(\bar b) = \bar b$.
We choose a sequence of $\widetilde e_i$ from $b \otimes u_{L\Lambda_0}$ to $\bar b \otimes u_{L\Lambda_0}$ and denote an element of this sequence by $g \otimes u_{L\Lambda_0}$.

For a step $\widetilde e_i$ in this sequence, we set $g' = \widetilde e_i g$.
Then $\varepsilon_i(g) > L\delta_{i0}$.
We have $\varepsilon_{i + 1}(\tau g) = \varepsilon_i(g) > L\delta_{i0} = \varphi_{i + 1}(1^L)$, with indices modulo $\ell$.
Hence
\[
	\tau g' \otimes 1^L = \widetilde e_{i + 1}(\tau g \otimes 1^L),
	\qquad v(g') \in \{v(g), \widetilde e_{i + 1}v(g)\}.
\]
For $1 \leq i < \ell - 1$, the elements $v(g)$ and $v(g')$ lie in the same classical component.
If $i = \ell - 1$ and $v(g) \neq v(g')$, we write $R(\tau g' \otimes 1^L) = y' \otimes v(g')$.
Then $v(g) = \widetilde f_0v(g')$ and $\varphi_0(v(g')) > \varepsilon_0(y')$, so
\[
	\varepsilon_0(v(g'))
	= \varepsilon_0(\tau g' \otimes 1^L)
	= L + \varepsilon_0(\tau g') > \mu_1.
\]
We apply \eqref{eq:zero-string-condition} to $v(g')$ and conclude that $v(g') \in Y$ if and only if $v(g) \in Y$.
For $i = 0$, the elements $v(g)$ and $v(g')$ lie in the same classical component.
Since $\varepsilon_0(g') = \varepsilon_0(g) - 1 \geq L > \mu_1$ and $\widetilde f_0g' = g$, we apply \eqref{eq:zero-string-condition} to $g'$ and conclude that $g' \in Y$ if and only if $g \in Y$.

Along the sequence, we obtain $b \in Y$ if and only if $\bar b \in Y$, and $v(b) \in Y$ if and only if $v(\bar b) \in Y$.
Since $v(\bar b) = \bar b$, we conclude that $b \in Y$ if and only if $v(b) \in Y$.
\end{proof}

\subsection{Strictly decreasing content}\label{subsec:strict-transport}

For strictly decreasing content, we characterize katabolizability using both smaller tableaux and prove preservation under $\Phi_{n, k}$ by induction on $n$.

For the remainder of this paper, $\Psi \subseteq \Delta_n^+$ is an upper root ideal with $\dvec = \dvec(\Psi) = (d_1, \ldots, d_{n - 1})$, and $T$ is a tableau with entries in $[n]$, partition content $\mu = (\mu_1, \ldots, \mu_n)$, and width at most $k$.
For $1 \leq m \leq n$, we construct a tableau with content $(\mu_{n - m + 1}, \ldots, \mu_n)$ by
\[
	\widehat T_m = w_0^{(m)} \cdot \bigl((w_0 \cdot T)^{\leq m}\bigr).
\]
The width of $\widehat T_m$ is at most $k$.
We set
\[
	\Psi_{> n - m} = \{(i, j) \in \Delta_m^+ : (n - m + i, n - m + j) \in \Psi\}.
\]
The set $\Psi_{> n - m}$ is an upper root ideal in $\Delta_m^+$ with $\dvec(\Psi_{> n - m}) = (d_{n - m + 1}, \ldots, d_{n - 1})$.

Since $R_{w_0^{(p + 1)}}(b \otimes y) = y' \otimes R_{w_0^{(p)}}v$,
\eqref{eq:R-recording} and \eqref{eq:recording-prefix} give
\begin{equation}\label{eq:R-tail-recording}
	Q(v) = w_0^{(p)} \cdot \bigl((w_0^{(p + 1)} \cdot Q(b \otimes y))^{\leq p}\bigr).
\end{equation}

\begin{lem}\label{lem:strict-suffix}
For $n \geq 2$ and $\mu_1 > \mu_2$, if $T$ is $\Psi$-katabolizable, then $\widehat T_{n - 1}$ is $\Psi_{> 1}$-katabolizable.
\end{lem}
\begin{proof}
The case $n = 2$ follows from $\Psi_{> 1} = \emptyset$ and $\widehat T_1 = Y_{(\mu_2)}$, so we assume that $n \geq 3$ and use alphabet $[n]$.
Since $\inv(T)$ is classically highest in $\mathcal D_{\Psi, \mu}$, it belongs to $Z_{\Psi, \mu}$.
Thus $\inv(T) = \tau x \otimes 1^{\mu_1}$ for some $x \in X_2^\Psi(\mu)$.
By \cref{prop:BMP-components,lem:dark-zero}, the set $Y = \mathcal D_{\Psi_{> 1}, (\mu_2, \ldots, \mu_n)}$ is a union of classical components and satisfies \eqref{eq:zero-string-condition} with $a = \mu_2$.
The recursion \eqref{eq:root-tail-stages} gives
\[
	X_2^\Psi(\mu) \subseteq
	\F_{w_0}\bigl(\tau(X_3^\Psi(\mu)) \otimes \{1^{\mu_2}\}\bigr) = Y.
\]
By \cref{prop:strict-carrier} with $L = \mu_1 > \mu_2$, we have $R(\tau x \otimes 1^{\mu_1}) = y' \otimes v(x)$ with $v(x) \in Y$.
By \eqref{eq:R-tail-recording}, we have
\[
\begin{aligned}
	Q(v(x))
	&= w_0^{(n - 1)} \cdot
		\bigl((w_0 \cdot Q(\tau x \otimes 1^{\mu_1}))^{\leq n - 1}\bigr)\\
	&= w_0^{(n - 1)} \cdot
		\bigl((w_0 \cdot T)^{\leq n - 1}\bigr)
	= \widehat T_{n - 1}.
\end{aligned}
\]
Since $v(x) \in Y$, $\widehat T_{n - 1}$ is $\Psi_{> 1}$-katabolizable.
\end{proof}

\begin{lem}\label{lem:strict-flags}
We assume that $A$ is an $\ell$-row tabloid with entries in $[m]$, has partition content $\mu$ with $\mu_1 > \mu_2$, and satisfies the first-row condition.
For $2 \leq r < m \leq \ell$ and $1 \leq d \leq \ell$, we set $\mu' = (\mu_{m - r + 1}, \ldots, \mu_m)$.
The following conditions are equivalent:
\begin{enumerate}
    \item $(w_0^{(m)} \cdot P(A))^{\leq r} = w_0^{(r)} \cdot Y_{\mu'}$;
    \item $(w_0^{(m - 1)} \cdot P(F_{d, \ell}A))^{\leq r} = w_0^{(r)} \cdot Y_{\mu'}$.
\end{enumerate}
\end{lem}
\begin{proof}
We consider
\[
	\mathcal L_r = \{b \in B^{(\mu_2, \ldots, \mu_m)} :
		(w_0^{(m - 1)} \cdot Q(b))^{\leq r} = w_0^{(r)} \cdot Y_{\mu'}\}.
\]
Since classical operators preserve $Q$, the set $\mathcal L_r$ is a union of classical components.
For $b \in B^{(\mu_2, \ldots, \mu_m)}$, the tableau $(w_0^{(m - 1)} \cdot Q(b))^{\leq r}$ is the recording tableau of the rightmost $r$ tensor factors of $R_{w_0^{(m - 1)}}b$.
Hence
\[
	R_{w_0^{(m - 1)}}\mathcal L_r
	= B' \otimes R_{w_0^{(r)}}\mathcal C_{Y_{\mu'}},
\]
where $B' = B^{1, \mu_2} \otimes \cdots \otimes B^{1, \mu_{m - r}}$, an empty tensor product when $r = m - 1$.

By \cref{lem:dark-zero}, the set $\mathcal C_{Y_{\mu'}} = \mathcal D_{\emptyset, \mu'}$ satisfies \eqref{eq:zero-string-condition} with $a = \mu_2$, since $\mu'_1 \leq \mu_2$.
Condition \eqref{eq:zero-string-condition} is preserved under affine crystal isomorphisms and tensoring on the left.
Therefore, $R_{w_0^{(m - 1)}}\mathcal L_r$ satisfies \eqref{eq:zero-string-condition} with $a = \mu_2$, and so does $\mathcal L_r$.

Since $A$ satisfies the first-row condition, $\inv(A) = \tau x \otimes 1^{\mu_1}$ for some $x \in B^{(\mu_2, \ldots, \mu_m)}$.
Let $R(\tau x \otimes 1^{\mu_1}) = y' \otimes v(x)$.
We have $Q(\inv(A)) = P(A)$ and $\tau^{-1}\pi\inv(A) = x$.
Therefore, with \eqref{eq:R-tail-recording}, we have
\[
	(w_0^{(m)} \cdot P(A))^{\leq m - 1}
	= w_0^{(m - 1)} \cdot Q(v(x))
\]
and
\[
	P(F_{d, \ell}A)
	= Q(\widetilde e_d^{\max}\cdots\widetilde e_{\ell - 1}^{\max}x)
	= Q(x).
\]
Note that conditions (1) and (2) are equivalent to $v(x) \in \mathcal L_r$ and $x \in \mathcal L_r$, respectively.
We apply \cref{prop:strict-carrier} with $L = \mu_1 > \mu_2$ and conclude that conditions (1) and (2) are equivalent.
\end{proof}

Let $A^{(0)}, \ldots, A^{(n - 1)}$ be the $\Psi$-katabolism sequence of $T$.
For $n \geq 3$, if $T^{\leq n - 1}$ is $(\Psi \cap \Delta_{n - 1}^+)$-katabolizable, then $A^{(1)}, \ldots, A^{(n - 2)}$ satisfy the first-row condition by \cref{lem:prefix-reduction}.
When $\Psi \neq \emptyset$ and $\mu_1 > \cdots > \mu_n$, we use the $\Psi_{> 1}$-katabolizability of $\widehat T_{n - 1}$ to show that $A^{(n - 1)}$ satisfies the first-row condition.
\begin{lem}\label{lem:trace-highest}
For $1 \leq t < n$, if $A^{(u)}$ satisfies the first-row condition for $0 \leq u < t$, then
\[
	\varepsilon_i\bigl(\inv(A^{(t)})\bigr) = 0
	\qquad(d_t \leq i < n).
\]
\end{lem}
\begin{proof}
It is enough to show that rows $d_t, \ldots, n$ of $A^{(t)}$ form a tableau.
We prove this by induction on $t$.
Since $A^{(t - 1)}$ satisfies the first-row condition, we have $(\kat A^{(t - 1)})_i = \dec(A^{(t - 1)}_{i + 1})$ for $d_t \leq i < n$.
For $t = 1$, $A^{(0)}$ is a tableau.
For $t \geq 2$, rows $d_{t - 1}, \ldots, n$ of $A^{(t - 1)}$ form a tableau by induction, and $d_t + 1 \geq d_{t - 1}$.
Since rows $d_t, \ldots, n - 1$ of $\kat(A^{(t - 1)})$ form a tableau, the operator $P_{d_t}\cdots P_{n - 1}$ column inserts $(\kat A^{(t - 1)})_n$ into this tableau.
Thus rows $d_t, \ldots, n$ of $A^{(t)}$ form a tableau.
\end{proof}

\begin{prop}\label{prop:two-sided-strict}
We assume that $n \geq 3$, $\Psi \neq \emptyset$, and $\mu_1 > \cdots > \mu_n$.
Then $T$ is $\Psi$-katabolizable if and only if $T^{\leq n - 1}$ is $(\Psi \cap \Delta_{n - 1}^+)$-katabolizable and $\widehat T_{n - 1}$ is $\Psi_{> 1}$-katabolizable.
\end{prop}
\begin{proof}
The ``only if'' direction follows from \cref{lem:prefix-reduction} with $m = n - 1$ and \cref{lem:strict-suffix}.
For the ``if'' direction, \cref{lem:prefix-reduction} with $m = n - 1$ gives the first-row condition for $A^{(t)}$ with $1 \leq t \leq n - 2$.
Since $A^{(0)} = T$ is semistandard, we only need to prove that $A^{(n - 1)}$ satisfies the first-row condition.
For the Hessenberg function $h$ of $\Psi$, we set $s = \max\{i : h_i < n\}$ and $r = n - s$.
Since $\Psi \neq \emptyset$, we have $h_1 < n$, so $1 \leq s \leq n - 1$.
Also, we have
\[
	d_s \leq r, \qquad d_{s + a} = r - a + 1\quad(1 \leq a < r),
	\qquad \Psi_{> s} = \emptyset.
\]
If $r = 1$, then $d_{n - 1} = 1$, so \cref{lem:trace-highest} with $t = n - 1$ gives $\inv(A^{(n - 1)}) = 1^{\mu_n}$, as required.

We now assume $r \geq 2$.
For $1 \leq j \leq m \leq n$, the definition of $\widehat T_m$ and $(w_0^{(m)})^2 = \id$ give
\begin{align*}
	\widehat{(\widehat T_m)}_j
	&= w_0^{(j)} \cdot \bigl((w_0^{(m)} \cdot \widehat T_m)^{\leq j}\bigr)\\
	&= w_0^{(j)} \cdot \bigl(((w_0 \cdot T)^{\leq m})^{\leq j}\bigr)\\
	&= w_0^{(j)} \cdot \bigl((w_0 \cdot T)^{\leq j}\bigr)
	= \widehat T_j.
\end{align*}
For $r < m \leq n - 1$, \cref{lem:strict-suffix} shows that if $\widehat T_m$ is $\Psi_{> n - m}$-katabolizable, then $\widehat T_{m - 1}$ is $\Psi_{> n - m + 1}$-katabolizable, since $(\Psi_{> n - m})_{> 1} = \Psi_{> n - m + 1}$.
Applying this for $m = n - 1, \ldots, r + 1$ shows that $\widehat T_r$ is katabolizable for the empty ideal $\Psi_{> s}$.
Since $\widehat T_r$ has content $\mu' = (\mu_{s + 1}, \ldots, \mu_n)$, we have $\widehat T_r = Y_{\mu'}$ and $(w_0 \cdot T)^{\leq r} = w_0^{(r)} \cdot Y_{\mu'}$.

For $0 \leq t < s$, the tabloid $A^{(t)}$ has content $(\mu_{t + 1}, \ldots, \mu_n)$ and satisfies the first-row condition.
\Cref{lem:strict-flags} with $m = n - t$ and $d = d_{t + 1}$ implies that if $(w_0^{(n - t)} \cdot P(A^{(t)}))^{\leq r} = w_0^{(r)} \cdot Y_{\mu'}$, then $(w_0^{(n - t - 1)} \cdot P(A^{(t + 1)}))^{\leq r} = w_0^{(r)} \cdot Y_{\mu'}$.
Therefore, starting with $A^{(0)} = T$, repeated application for $t = 0, \ldots, s - 1$ gives $(w_0^{(r)} \cdot P(A^{(s)}))^{\leq r} = w_0^{(r)} \cdot Y_{\mu'}$.
Since $A^{(s)}$ has entries in $[r]$, we obtain $P(A^{(s)}) = Y_{\mu'}$.

We claim that every entry $r$ of $A^{(s)}$ lies in row $r$.
For $2 \leq j < r$, the tabloid $A^{(s + j - 1)}$ satisfies the first-row condition since $s + j - 1 \leq n - 2$.
Since $d_{s + a} = r - a + 1$ for $1 \leq a < r$, \cref{lem:top-rows} with $t = s$ shows that every entry $j$ of $A^{(s)}$ lies in row $j$ for these $j$.
Since $A^{(s)}$ also satisfies the first-row condition, we have
\[
	b := \inv(A^{(s)}) = z \otimes b', \qquad
	b' = (r - 1)^{\mu'_{r - 1}} \otimes \cdots \otimes 1^{\mu'_1},
\]
where $z \in B^{1, \mu'_r}$.
Since $Q(b) = P(A^{(s)}) = Y_{\mu'}$, we have $b \in \mathcal C_{Y_{\mu'}}$.
The element $b$ is obtained from $r^{\mu'_r} \otimes b'$ by applying $\widetilde f_i$ with $1 \leq i < n$, so $\wt_i(z) = 0$ for $1 \leq i < r$.
Since $d_s \leq r$, \cref{lem:trace-highest} gives $\varepsilon_i(b) = 0$ for $r \leq i < n$.
For $r \leq i < n$, we have $\varepsilon_i(b') = \varphi_i(b') = 0$, so
\[
	\wt_{i + 1}(z) = \varepsilon_i(z) = \varepsilon_i(b) = 0.
\]
Thus $z = r^{\mu'_r}$, so $A^{(n - 1)}$ satisfies the first-row condition by \cref{lem:top-rows} with $t = s$ and $j = r$.
Therefore $T$ is $\Psi$-katabolizable.
\end{proof}

For a tableau $V$ with entries in $[n]$ and width at most $k$, and for $1 \leq j \leq n$, column complementation satisfies
\begin{equation}\label{eq:D-restriction}
	(D_{n, k}V)^{\leq j} = D_{j, k}(V^{\leq j}),
\end{equation}
since $([n] \setminus C) \cap [j] = [j] \setminus (C \cap [j])$ for each column $C$.

\begin{prop}\label{prop:strict-transport}
We assume that $\mu_1 > \cdots > \mu_n$.
Then $T$ is $\Psi$-katabolizable if and only if $\Phi_{n, k}T$ is $\Psi^\dagger$-katabolizable.
\end{prop}
\begin{proof}
Since $\Phi_{n, k}^2 = \id$ and $(\Psi^\dagger)^\dagger = \Psi$, it suffices to show that $\Phi_{n, k}T$ is $\Psi^\dagger$-katabolizable if $T$ is $\Psi$-katabolizable.
We use induction on $n$.
The case $n = 1$ is immediate.
For $n \geq 2$ and $\Psi = \emptyset$, \cref{lem:top-rows} gives $T = Y_\mu$, so $\Phi_{n, k}T = Y_{\mu^\perp}$ is $\emptyset$-katabolizable.
For $n = 2$ and $\Psi \neq \emptyset$, $\Psi = \Psi^\dagger = \Delta_2^+$ implies that every tableau is $\Psi$-katabolizable.

For $n \geq 3$ and $\Psi \neq \emptyset$, let $U = \Phi_{n, k}T$.
By \eqref{eq:D-restriction} and \cref{lem:D-crystal}, we have
\begin{align*}
	U^{\leq n - 1}
	&= (D_{n, k}(w_0 \cdot T))^{\leq n - 1}\\
	&= D_{n - 1, k}((w_0 \cdot T)^{\leq n - 1})
	= \Phi_{n - 1, k}\widehat T_{n - 1}, \\
	\widehat U_{n - 1}
	&= w_0^{(n - 1)} \cdot \bigl((D_{n, k}T)^{\leq n - 1}\bigr)\\
	&= w_0^{(n - 1)} \cdot D_{n - 1, k}(T^{\leq n - 1})
	= \Phi_{n - 1, k}(T^{\leq n - 1}).
\end{align*}

Note that $T^{\leq n - 1}$ is $(\Psi \cap \Delta_{n - 1}^+)$-katabolizable and $\widehat T_{n - 1}$ is $\Psi_{> 1}$-katabolizable.
Antidiagonal reflection gives
\[
	\Psi^\dagger \cap \Delta_{n - 1}^+ = (\Psi_{> 1})^\dagger,
	\qquad
	(\Psi^\dagger)_{> 1} = (\Psi \cap \Delta_{n - 1}^+)^\dagger.
\]
By the induction hypothesis, $U^{\leq n - 1}$ is $(\Psi^\dagger \cap \Delta_{n - 1}^+)$-katabolizable and $\widehat U_{n - 1}$ is $(\Psi^\dagger)_{> 1}$-katabolizable.
Thus $U$ is $\Psi^\dagger$-katabolizable by \cref{prop:two-sided-strict}.
\end{proof}

\begin{rmk}\label{rem:non-strict-criterion}
The converse in \cref{prop:two-sided-strict} can fail when the content is not strictly decreasing.
For
\[
	\Psi = \{(1, 3), (1, 4), (2, 3), (2, 4)\}, \qquad
	T = \begin{ytableau}1&3&4\\2\end{ytableau},
\]
we have $\mu = (1, 1, 1, 1)$ and $\dvec = (2, 1, 2)$.
The tableau
\[
	T^{\leq 3} = \widehat T_3 = \begin{ytableau}1&3\\2\end{ytableau}
\]
is katabolizable for both $\Psi \cap \Delta_3^+ = \{(1, 3), (2, 3)\}$ and $\Psi_{> 1} = \{(1, 2), (1, 3)\}$.
However, the $\Psi$-katabolism sequence of $T$ is
\[
	\langle134; 2; \emptyset; \emptyset\rangle
	\longrightarrow \langle1; 23; \emptyset; \emptyset\rangle
	\longrightarrow \langle12; \emptyset; \emptyset; \emptyset\rangle
	\longrightarrow \langle\emptyset; 1; \emptyset; \emptyset\rangle.
\]
The last tabloid has a $1$ in the second row, so $T$ is not $\Psi$-katabolizable.
\end{rmk}

\subsection{Arbitrary partition content}\label{subsec:finite-transport}

We extend the strictly decreasing case to arbitrary partition content using the piecewise-linearity of tableau and crystal operations.

For a tabloid $A$ and $b = b_p \otimes \cdots \otimes b_1 \in B^\alpha$, we define $[A]_{i, j} = \wt_j(A_i)$ and $[b]_{i, j} = \wt_j(b_i)$.

A \emph{rational polyhedral cone} in $\mathbb R^d$ is a set
\[
	C = \{x \in \mathbb R^d : Hx \geq 0\},
	\qquad H \in \mathbb Q^{r \times d}.
\]
We consider piecewise-linear maps
\[
	f : \bigcup_{j = 1}^{c}C_j \longrightarrow \mathbb R^e,
	\qquad f(x) = M_jx\quad(x \in C_j),
\]
where each $C_j$ is a rational polyhedral cone, $M_j \in \mathbb Z^{e \times d}$, and the formulas agree on intersections.
We use their restrictions to the lattice points $\bigcup_{j = 1}^{c}(C_j \cap \mathbb Z^d)$.
Piecewise-linear maps are closed under finite compositions.

We have $[\inv(A)] = [A]^T$.
The maps $[b] \mapsto [\tau_\ell^{\pm1}b]$, $[b] \mapsto [\pi b]$, and $(a, [b]) \mapsto [b \otimes 1^a]$ are linear, where $a \in \mathbb Z_{\geq 0}$.
The functions $\varepsilon_i(b)$ and $\varphi_i(b)$ are piecewise-linear in $[b]$ by \eqref{eq:tensor-epsilon} and \eqref{eq:tensor-phi}.
For a tableau $V$ and $j \geq 1$, the partition $\sh_j(V) = \sh(V^{\leq j})$ is padded with zeros to $j$ parts.
For $V$ with entries in $[n]$ and $k \geq \width(V)$, we have
\[
	\sh_j(D_{n, k}V)_i = k - \sh_j(V)_{j + 1 - i}
	\qquad(1 \leq i \leq j \leq n).
\]
Since $\sh_j(V)_i = \sum_{a = 1}^j[V]_{i, a}$ and $[V]_{i, j} = \sh_j(V)_i - \sh_{j - 1}(V)_i$, the map $([V], k) \mapsto [D_{n, k}V]$ is linear.

For $1 \leq i < \ell$, we consider the maps
\begin{equation}\label{eq:coordinate-crystal-maps}
	(a, [b]) \longmapsto \bigl[\widetilde e_i^{\,\min\{a, \varepsilon_i(b)\}}b\bigr],
	\qquad
	(a, [b]) \longmapsto \bigl[\widetilde f_i^{\,\min\{a, \varphi_i(b)\}}b\bigr],
\end{equation}
where $a \in \mathbb Z_{\geq 0}$, $\alpha \in \mathbb Z_{\geq 0}^p$, and $b \in B^\alpha$ has letters in $[\ell]$.

\begin{lem}\label{lem:min-max-operations}
The maps in \eqref{eq:coordinate-crystal-maps} are piecewise-linear.
\end{lem}
\begin{proof}
We use induction on the number of tensor factors.
For $b \in B^{1, s}$, both maps are piecewise-linear by the definitions of $\widetilde e_i$ and $\widetilde f_i$.
For $b = x \otimes y$ with $y \in B^{1, s}$, we set $a_f = \min\{a, \varphi_i(x \otimes y)\}$ and $a_e = \min\{a, \varepsilon_i(x \otimes y)\}$.
Then
\begin{align*}
	\widetilde f_i^{\,a_f}(x \otimes y)
		&= (\widetilde f_i^{\,a_f - r_f}x) \otimes (\widetilde f_i^{\,r_f}y),
	&r_f&= \min\{a_f, \max(0, \varphi_i(y) - \varepsilon_i(x))\},
	\\
	\widetilde e_i^{\,a_e}(x \otimes y)
		&= (\widetilde e_i^{\,r_e}x) \otimes (\widetilde e_i^{\,a_e - r_e}y),
	&r_e&= \min\{a_e, \max(0, \varepsilon_i(x) - \varphi_i(y))\}.
\end{align*}
The induction hypothesis shows that both maps are piecewise-linear.
\end{proof}

Taking $a = \varepsilon_i(b)$ in \cref{lem:min-max-operations} shows that $[b] \mapsto [\widetilde e_i^{\max}b]$ is piecewise-linear.
For the LS action, we view $V$ as the tensor product of its rows, from bottom to top.
The map $[V] \mapsto [s_i \cdot V]$ is then piecewise-linear by the $i$-string description in \cref{subsec:crystal-prelim}.
Hence the maps
\[
	([V], k) \longmapsto [\Phi_{n, k}V],
	\qquad [A] \longmapsto [F_{d, \ell}A]
	\quad(1 \leq d \leq \ell)
\]
are piecewise-linear, where $A$ has $\ell$ rows; the assertion for $F_{d, \ell}$ follows from \eqref{eq:kat-crystal}.

\begin{proof}[Proof of \cref{thm:main}]
The cases $n \leq 2$ follow as in the proof of \cref{prop:strict-transport}.
We assume that $n \geq 3$ and $T$ is $\Psi$-katabolizable.
With alphabet $[n]$, we expand each $\F_w$ in \eqref{eq:root-tail-stages} as $\F_{j_1}\cdots\F_{j_r}$ using a fixed reduced expression $w = s_{j_1}\cdots s_{j_r}$.
For the final operator, we use $w_0 = s_{i_1}\cdots s_{i_s}$.
Let $M = s + \sum_{j = 1}^{n - 1}(n - d_j)$ be the total number of occurrences of $\F_i$ in the recursion.
We assign exponents $\alpha_1, \ldots, \alpha_M$ to these occurrences in their order of application.

Let
\[
	\Omega
	= \{(\nu, \alpha, \kappa) \in \mathbb Z^n \times \mathbb Z_{\geq 0}^{M} \times \mathbb Z_{\geq 0} :
		\nu_1 \geq \cdots \geq \nu_n \geq 0\}.
\]
For $\beta = (\nu, \alpha, \kappa) \in \Omega$, we choose $b(\beta) \in \mathcal D_{\Psi, \nu}$ by following the recursion \eqref{eq:root-tail-stages} with content $\nu$.
At the $j$th occurrence, say $\F_i$, we choose the element $\widetilde f_i^{\,\min\{\alpha_j, \varphi_i(b)\}}b$ from $\F_i\{b\}$, where $b$ is the current tensor.
We define $T(\beta)$ by
\[
	\inv(T(\beta))
	= \widetilde e_{i_1}^{\max}\cdots\widetilde e_{i_s}^{\max}b(\beta).
\]
$\widetilde e_{i_1}^{\max}\cdots\widetilde e_{i_s}^{\max}b(\beta)$ is classically highest~\cite[Proposition~6.11(iv)]{BlaMorPun2024}, so $T(\beta)$ is $\Psi$-katabolizable with content $\nu$ by \cref{prop:BMP-components}.
We set $k(\beta) = \width(T(\beta)) + \kappa = \sum_{j = 1}^{n}[T(\beta)]_{1, j} + \kappa$.

Since $\inv(T) \in \mathcal D_{\Psi, \mu}$, we choose $\alpha^0$ such that
\[
	\beta_0 = (\mu, \alpha^0, k - \width(T)) \in \Omega,
	\qquad b(\beta_0) = \inv(T).
\]
$\inv(T)$ is classically highest, so $T(\beta_0) = T$ and $k(\beta_0) = k$.

For each $\beta \in \Omega$, let $(A^{(t)}(\beta))_{t = 0}^{n - 1}$ be the $\Psi^\dagger$-katabolism sequence of $\Phi_{n, k(\beta)}T(\beta)$, computed through $t = n - 1$ even if a first-row condition fails.
We define $\eta : \Omega \to \mathbb Z_{\geq 0}$ by
\[
	\eta(\beta)
	= \sum_{t = 0}^{n - 1}\sum_{r = 2}^{n}[A^{(t)}(\beta)]_{r, 1}.
\]
Thus $\Phi_{n, k}T$ is $\Psi^\dagger$-katabolizable if and only if $\eta(\beta_0) = 0$.

By \cref{lem:min-max-operations} and closure under composition, $\eta$ is piecewise-linear in $\beta$.
Hence there are finitely many rational polyhedral cones $C_1, \ldots, C_c \subseteq \mathbb R^{n + M + 1}$ such that
\[
	\Omega
	= \bigcup_{j = 1}^{c}\bigl(C_j \cap \mathbb Z^{n + M + 1}\bigr),
\]
and $\eta$ is linear on the lattice points of each $C_j$.
With $\rho = (n - 1, \ldots, 0)$, we set
\[
	\gamma = (\rho, 0^M, 0), \qquad
	\beta_t = t\beta_0 + \gamma
	\qquad(t \in \mathbb Z_{\geq 1}).
\]
For every $t \geq 1$, we have $\beta_t \in \Omega$ and $\wt(T(\beta_t)) = t\mu + \rho$, which is strictly decreasing.
Therefore, \cref{prop:strict-transport} gives $\eta(\beta_t) = 0$ for every $t \geq 1$.

Since the cover is finite, some cone $C$ contains $\beta_t$ for infinitely many $t$.
We write $C = \{x \in \mathbb R^{n + M + 1} : Hx \geq 0\}$.
Then
\[
	0 \leq H\beta_t = tH\beta_0 + H\gamma
\]
for arbitrarily large integers $t$.
Since $H\gamma$ is fixed, each entry of $H\beta_0$ must be nonnegative.
Thus $\beta_0 \in C$.

For any $t$ with $\beta_t \in C$, closure of $C$ under addition gives $\beta_{t + 1} = \beta_t + \beta_0 \in C$.
Since $\eta$ is linear on the lattice points of $C$, we obtain
\[
	\eta(\beta_0)
	= \eta(\beta_{t + 1}) - \eta(\beta_t)
	= 0.
\]
Hence $\Phi_{n, k}T$ is $\Psi^\dagger$-katabolizable.
\end{proof}

\begin{rmk}
Combining \cref{prop:D,lem:D-crystal,prop:ambient-charge,thm:main,thm:BMP} gives a bijective proof of \eqref{eq:intro-catalan-symmetry}.
\end{rmk}

\begin{exmp}\label{ex:charge-not-enough}
The identity map preserves charge but need not send $\Psi$-katabolizable tableaux to $\Psi^\dagger$-katabolizable tableaux.
For $n = 4$ and $k = 3$, we use the ideal $\Psi$ from \cref{ex:root-ideal-dyck}.
Then $\Psi^\dagger = \Delta_4^+ \setminus \{(2, 3), (3, 4)\}$, with $\dvec(\Psi) = (2, 2, 1)$ and $\dvec(\Psi^\dagger) = (1, 2, 2)$.

For $\lambda = (3, 2, 1, 0) = \lambda^\perp$ and $\mu = (2, 2, 1, 1) = \mu^\perp$, the set $\SSYT_4(\lambda, \mu)$ consists of
\[
	T_1 = \begin{ytableau}1&1&2\\2&3\\4\end{ytableau}, \quad
	T_2 = \begin{ytableau}1&1&2\\2&4\\3\end{ytableau}, \quad
	T_3 = \begin{ytableau}1&1&3\\2&2\\4\end{ytableau}, \quad
	T_4 = \begin{ytableau}1&1&4\\2&2\\3\end{ytableau},
\]
with charges $3, 2, 2, 1$, respectively.
The $\Psi$-katabolism sequence of $T_3$ and the $\Psi^\dagger$-katabolism sequence of $T_2$ are
\[
\begin{aligned}
	\langle113; 22; 4; \emptyset\rangle
	& \xrightarrow{F_{2, 4}} \langle11; 23; \emptyset; \emptyset\rangle
	\xrightarrow{F_{2, 4}} \langle12; \emptyset; \emptyset; \emptyset\rangle
	\xrightarrow{F_{1, 4}} \langle1; \emptyset; \emptyset; \emptyset\rangle, \\
	\langle112; 24; 3; \emptyset\rangle
	& \xrightarrow{F_{1, 4}} \langle113; 2; \emptyset; \emptyset\rangle
	\xrightarrow{F_{2, 4}} \langle1; 2; \emptyset; \emptyset\rangle
	\xrightarrow{F_{2, 4}} \langle1; \emptyset; \emptyset; \emptyset\rangle.
\end{aligned}
\]
The tableau $T_2$ is not $\Psi$-katabolizable, since $F_{2, 4}T_2 = \langle13; 12; \emptyset; \emptyset\rangle$.
The $\Psi^\dagger$- and $\Psi$-katabolism sequences of $T_3$ agree after the first and second steps, since $P_1$ fixes $\langle11; 23; \emptyset; \emptyset\rangle$.
In the $\Psi^\dagger$-katabolism sequence, the last step $F_{2, 4}$ gives $\langle\emptyset; 1; \emptyset; \emptyset\rangle$, so $T_3$ is not $\Psi^\dagger$-katabolizable.
Similar computations for $T_1$ and $T_4$ give
\[
	\tset_\Psi(\lambda, \mu) = \{T_3, T_4\},
	\qquad
	\tset_{\Psi^\dagger}(\lambda, \mu) = \{T_2, T_4\}.
\]
By \cref{thm:BMP}, $K^\Psi_{\lambda, \mu}(q) = K^{\Psi^\dagger}_{\lambda, \mu}(q) = q + q^2$.
The involution $\Phi_{4, 3}$ satisfies
\[
	\Phi_{4, 3}T_2 = T_3, \qquad
	\Phi_{4, 3}T_3 = T_2, \qquad
	\Phi_{4, 3}T_i = T_i\quad(i = 1, 4),
\]
so it gives the required bijection.
\end{exmp}

\section*{Acknowledgements}
The author thanks Jaeseong Oh for helpful discussions.
H. Choi was supported by the Basic Science Research Program through the National Research Foundation of Korea (NRF) funded by the Ministry of Education (No.~RS-2026-25561840).

\section*{AI declaration}
The author used ChatGPT (OpenAI) for literature searches and citation checks, coding and computations in SageMath, figure generation, and drafting and polishing the manuscript.
ChatGPT also assisted with developing and checking proof arguments based on the author's ideas and directions.
The author takes full responsibility for the content of this paper.

\begingroup
\emergencystretch=1em
\renewcommand{\bibfont}{\small}
\setlength{\bibitemsep}{1pt}
\setlength{\bibparsep}{0pt}
\printbibliography

@misc{Bla2020,
      title={The {DARK} side of generalized {Demazure} crystals}, 
      author={Jonah Blasiak},
      year={2020},
      eprint={2007.04888},
      options={eprint=true},
      archivePrefix={arXiv},
      primaryClass={math.QA},
      url={https://arxiv.org/abs/2007.04888}, 
}

@article {BlaMorPun2024,
    AUTHOR = {Blasiak, Jonah and Morse, Jennifer and Pun, Anna},
     TITLE = {Demazure crystals and the {S}chur positivity of {C}atalan
              functions},
   JOURNAL = {Invent. Math.},
  FJOURNAL = {Inventiones Mathematicae},
    VOLUME = {236},
      YEAR = {2024},
    NUMBER = {2},
     PAGES = {483-547},
      ISSN = {0020-9910,1432-1297},
   MRCLASS = {05E05 (14N15 33E99 81R50)},
  MRNUMBER = {4728238},
MRREVIEWER = {Domenico\ Senato},
       DOI = {10.1007/s00222-024-01237-5},
       URL = {https://doi.org/10.1007/s00222-024-01237-5},
}

@article {BlaMorPunSum2019,
    AUTHOR = {Blasiak, Jonah and Morse, Jennifer and Pun, Anna and Summers,
              Daniel},
     TITLE = {Catalan functions and {$k$}-{S}chur positivity},
   JOURNAL = {J. Amer. Math. Soc.},
  FJOURNAL = {Journal of the American Mathematical Society},
    VOLUME = {32},
      YEAR = {2019},
    NUMBER = {4},
     PAGES = {921-963},
      ISSN = {0894-0347,1088-6834},
   MRCLASS = {05E05 (05E10)},
  MRNUMBER = {4013737},
MRREVIEWER = {Tanja\ Stojadinovi\'c},
       DOI = {10.1090/jams/921},
       URL = {https://doi.org/10.1090/jams/921},
}

@article {BriOreRos2015,
    AUTHOR = {Briand, Emmanuel and Orellana, Rosa and Rosas, Mercedes},
     TITLE = {Rectangular symmetries for coefficients of symmetric
              functions},
   JOURNAL = {Electron. J. Combin.},
  FJOURNAL = {Electronic Journal of Combinatorics},
    VOLUME = {22},
      YEAR = {2015},
    NUMBER = {3},
     PAGES = {Paper 3.15, 18},
      ISSN = {1077-8926},
   MRCLASS = {05E05 (05E10)},
  MRNUMBER = {3386516},
MRREVIEWER = {David\ J.\ Grabiner},
       DOI = {10.37236/4808},
       URL = {https://doi.org/10.37236/4808},
}

@book {BumSch2017,
    AUTHOR = {Bump, Daniel and Schilling, Anne},
     TITLE = {Crystal bases},
      NOTE = {Representations and combinatorics},
 PUBLISHER = {World Scientific Publishing Co. Pte. Ltd., Hackensack, NJ},
      YEAR = {2017},
     PAGES = {xii+279},
      ISBN = {978-981-4733-44-1},
   MRCLASS = {05-01 (05E10 14T05 17B10)},
  MRNUMBER = {3642318},
       DOI = {10.1142/9876},
       URL = {https://doi.org/10.1142/9876},
}

@book {Ful1997,
    AUTHOR = {Fulton, William},
     TITLE = {Young tableaux},
    SERIES = {London Mathematical Society Student Texts},
    VOLUME = {35},
      NOTE = {With applications to representation theory and geometry},
 PUBLISHER = {Cambridge University Press, Cambridge},
      YEAR = {1997},
     PAGES = {x+260},
      ISBN = {0-521-56144-2; 0-521-56724-6},
   MRCLASS = {05E10 (05E05 05E15 14M15 20G05)},
  MRNUMBER = {1464693},
MRREVIEWER = {Tadeusz\ J\'ozefiak},
}

@book {HonKan2002,
    AUTHOR = {Hong, Jin and Kang, Seok-Jin},
     TITLE = {Introduction to quantum groups and crystal bases},
    SERIES = {Graduate Studies in Mathematics},
    VOLUME = {42},
 PUBLISHER = {American Mathematical Society, Providence, RI},
      YEAR = {2002},
     PAGES = {xviii+307},
      ISBN = {0-8218-2874-6},
   MRCLASS = {17B37 (05E15 81R50 82B23)},
  MRNUMBER = {1881971},
MRREVIEWER = {Anne\ Schilling},
       DOI = {10.1090/gsm/042},
       URL = {https://doi.org/10.1090/gsm/042},
}

@incollection {KanKasMisMiwNakNak1992,
    AUTHOR = {Kang, Seok-Jin and Kashiwara, Masaki and Misra, Kailash C. and
              Miwa, Tetsuji and Nakashima, Toshiki and Nakayashiki, Atsushi},
     TITLE = {Affine crystals and vertex models},
 BOOKTITLE = {Infinite analysis, {P}art {A}, {B} ({K}yoto, 1991)},
    SERIES = {Adv. Ser. Math. Phys.},
    VOLUME = {16},
     PAGES = {449-484},
 PUBLISHER = {World Sci. Publ., River Edge, NJ},
      YEAR = {1992},
      ISBN = {981-02-0955-X},
   MRCLASS = {17B37 (17B67 81R10 82B23)},
  MRNUMBER = {1187560},
MRREVIEWER = {Peter\ J.\ Forrester},
       DOI = {10.1142/s0217751x92003896},
       URL = {https://doi.org/10.1142/s0217751x92003896},
}

@article {Kas1991,
    AUTHOR = {Kashiwara, M.},
     TITLE = {On crystal bases of the {$Q$}-analogue of universal enveloping
              algebras},
   JOURNAL = {Duke Math. J.},
  FJOURNAL = {Duke Mathematical Journal},
    VOLUME = {63},
      YEAR = {1991},
    NUMBER = {2},
     PAGES = {465-516},
      ISSN = {0012-7094,1547-7398},
   MRCLASS = {17B37 (17B10 17B67)},
  MRNUMBER = {1115118},
MRREVIEWER = {Kailash\ C.\ Misra},
       DOI = {10.1215/S0012-7094-91-06321-0},
       URL = {https://doi.org/10.1215/S0012-7094-91-06321-0},
}

@article {Kas1993,
    AUTHOR = {Kashiwara, Masaki},
     TITLE = {The crystal base and {L}ittelmann's refined {D}emazure
              character formula},
   JOURNAL = {Duke Math. J.},
  FJOURNAL = {Duke Mathematical Journal},
    VOLUME = {71},
      YEAR = {1993},
    NUMBER = {3},
     PAGES = {839-858},
      ISSN = {0012-7094,1547-7398},
   MRCLASS = {17B37 (17B10 17B67)},
  MRNUMBER = {1240605},
MRREVIEWER = {Kailash\ C.\ Misra},
       DOI = {10.1215/S0012-7094-93-07131-1},
       URL = {https://doi.org/10.1215/S0012-7094-93-07131-1},
}

@article {KirShi2002,
    AUTHOR = {Kirillov, Anatol N. and Shimozono, Mark},
     TITLE = {A generalization of the {K}ostka-{F}oulkes polynomials},
   JOURNAL = {J. Algebraic Combin.},
  FJOURNAL = {Journal of Algebraic Combinatorics. An International Journal},
    VOLUME = {15},
      YEAR = {2002},
    NUMBER = {1},
     PAGES = {27-69},
      ISSN = {0925-9899,1572-9192},
   MRCLASS = {05E15 (05A30 20G05 82B23)},
  MRNUMBER = {1879658},
MRREVIEWER = {Anne\ Schilling},
       DOI = {10.1023/A:1013269131974},
       URL = {https://doi.org/10.1023/A:1013269131974},
}

@article {Knu1970,
    AUTHOR = {Knuth, Donald E.},
     TITLE = {Permutations, matrices, and generalized {Y}oung tableaux},
   JOURNAL = {Pacific J. Math.},
  FJOURNAL = {Pacific Journal of Mathematics},
    VOLUME = {34},
      YEAR = {1970},
     PAGES = {709-727},
      ISSN = {0030-8730,1945-5844},
   MRCLASS = {05.30},
  MRNUMBER = {272654},
MRREVIEWER = {M.\ Doob},
       URL = {http://projecteuclid.org/euclid.pjm/1102971948},
}

@article {LasSch1978,
    AUTHOR = {Lascoux, Alain and Sch\"utzenberger, Marcel-Paul},
     TITLE = {Sur une conjecture de {H}. {O}. {F}oulkes},
   JOURNAL = {C. R. Acad. Sci. Paris S\'er. A-B},
  FJOURNAL = {Comptes Rendus Hebdomadaires des S\'eances de l'Acad\'emie des
              Sciences. S\'eries A et B},
    VOLUME = {286},
      YEAR = {1978},
    NUMBER = {7},
     PAGES = {A323-A324},
      ISSN = {0151-0509},
   MRCLASS = {20C30 (14M15)},
  MRNUMBER = {472993},
}

@incollection {LasSch1981,
    AUTHOR = {Lascoux, Alain and Sch\"utzenberger, Marcel-P.},
     TITLE = {Le mono\"ide plaxique},
 BOOKTITLE = {Noncommutative structures in algebra and geometric
              combinatorics ({N}aples, 1978)},
    SERIES = {Quad. ``Ricerca Sci.''},
    VOLUME = {109},
     PAGES = {129-156},
 PUBLISHER = {CNR, Rome},
      YEAR = {1981},
   MRCLASS = {20C30 (05A99 20M10)},
  MRNUMBER = {646486},
MRREVIEWER = {Hubert\ Grassmann},
}

@article {Lus1990,
    AUTHOR = {Lusztig, G.},
     TITLE = {Canonical bases arising from quantized enveloping algebras},
   JOURNAL = {J. Amer. Math. Soc.},
  FJOURNAL = {Journal of the American Mathematical Society},
    VOLUME = {3},
      YEAR = {1990},
    NUMBER = {2},
     PAGES = {447-498},
      ISSN = {0894-0347,1088-6834},
   MRCLASS = {17B35 (16A64)},
  MRNUMBER = {1035415},
MRREVIEWER = {Ya.\ S.\ So\u ibel\cprime man},
       DOI = {10.2307/1990961},
       URL = {https://doi.org/10.2307/1990961},
}

@article {NakYam1997,
    AUTHOR = {Nakayashiki, Atsushi and Yamada, Yasuhiko},
     TITLE = {Kostka polynomials and energy functions in solvable lattice
              models},
   JOURNAL = {Selecta Math. (N.S.)},
  FJOURNAL = {Selecta Mathematica. New Series},
    VOLUME = {3},
      YEAR = {1997},
    NUMBER = {4},
     PAGES = {547-599},
      ISSN = {1022-1824,1420-9020},
   MRCLASS = {05E05 (17B10 17B37 81R50 82B23)},
  MRNUMBER = {1613527},
MRREVIEWER = {Kailash\ C.\ Misra},
       DOI = {10.1007/s000290050020},
       URL = {https://doi.org/10.1007/s000290050020},
}

@article {Nao2013,
    AUTHOR = {Naoi, Katsuyuki},
     TITLE = {Demazure crystals and tensor products of perfect
              {K}irillov-{R}eshetikhin crystals with various levels},
   JOURNAL = {J. Algebra},
  FJOURNAL = {Journal of Algebra},
    VOLUME = {374},
      YEAR = {2013},
     PAGES = {1-26},
      ISSN = {0021-8693,1090-266X},
   MRCLASS = {17B37},
  MRNUMBER = {2998791},
MRREVIEWER = {Peter\ W.\ Tingley},
       DOI = {10.1016/j.jalgebra.2012.10.020},
       URL = {https://doi.org/10.1016/j.jalgebra.2012.10.020},
}

@article {SchTin2012,
    AUTHOR = {Schilling, Anne and Tingley, Peter},
     TITLE = {Demazure crystals, {K}irillov-{R}eshetikhin crystals, and the
              energy function},
   JOURNAL = {Electron. J. Combin.},
  FJOURNAL = {Electronic Journal of Combinatorics},
    VOLUME = {19},
      YEAR = {2012},
    NUMBER = {2},
     PAGES = {Paper 4, 42},
      ISSN = {1077-8926},
   MRCLASS = {17B37 (05E05 05E10)},
  MRNUMBER = {2923717},
MRREVIEWER = {Hiro-Fumi\ Yamada},
       DOI = {10.37236/2184},
       URL = {https://doi.org/10.37236/2184},
}

@article {Shi2002Affine,
    AUTHOR = {Shimozono, Mark},
     TITLE = {Affine type {A} crystal structure on tensor products of
              rectangles, {D}emazure characters, and nilpotent varieties},
   JOURNAL = {J. Algebraic Combin.},
  FJOURNAL = {Journal of Algebraic Combinatorics. An International Journal},
    VOLUME = {15},
      YEAR = {2002},
    NUMBER = {2},
     PAGES = {151-187},
      ISSN = {0925-9899,1572-9192},
   MRCLASS = {17B10 (05E10)},
  MRNUMBER = {1887233},
MRREVIEWER = {Stefano\ Capparelli},
       DOI = {10.1023/A:1013894920862},
       URL = {https://doi.org/10.1023/A:1013894920862},
}

@article {ShiWey2000,
    AUTHOR = {Shimozono, Mark and Weyman, Jerzy},
     TITLE = {Graded characters of modules supported in the closure of a
              nilpotent conjugacy class},
   JOURNAL = {European J. Combin.},
  FJOURNAL = {European Journal of Combinatorics},
    VOLUME = {21},
      YEAR = {2000},
    NUMBER = {2},
     PAGES = {257-288},
      ISSN = {0195-6698,1095-9971},
   MRCLASS = {05E15 (20G05)},
  MRNUMBER = {1742440},
       DOI = {10.1006/eujc.1999.0344},
       URL = {https://doi.org/10.1006/eujc.1999.0344},
}

@book {Sta1999,
    AUTHOR = {Stanley, Richard P.},
     TITLE = {Enumerative combinatorics. {V}ol. 2},
    SERIES = {Cambridge Studies in Advanced Mathematics},
    VOLUME = {62},
      NOTE = {With a foreword by Gian-Carlo Rota and appendix 1 by Sergey
              Fomin},
 PUBLISHER = {Cambridge University Press, Cambridge},
      YEAR = {1999},
     PAGES = {xii+581},
      ISBN = {0-521-56069-1; 0-521-78987-7},
   MRCLASS = {05A15 (05-02 05E05 05E10 68R05)},
  MRNUMBER = {1676282},
MRREVIEWER = {Ira\ Gessel},
       DOI = {10.1017/CBO9780511609589},
       URL = {https://doi.org/10.1017/CBO9780511609589},
}
\endgroup

\end{document}